\documentclass[12pt,reqno]{amsart}

\usepackage{amsmath,amssymb,amsthm, geometry}
\usepackage{setspace}
\usepackage{mathrsfs}
\usepackage{graphicx}
\usepackage{enumerate}
\usepackage[colorlinks=false, urlcolor=blue, final]{hyperref}
\usepackage[all]{xy}
\usepackage{pdfsync}

\newtheorem{thm}{Theorem}[section]

\newtheorem{lem}[thm]{Lemma}
\newtheorem{cor}[thm]{Corollary}
\newtheorem{prop}[thm]{Proposition}

\theoremstyle{definition}
\newtheorem{defn}[thm]{Definition}
\newtheorem{rmk}[thm]{Remark}

\allowdisplaybreaks

\theoremstyle{remark}
\newtheorem{exs}[thm]{Examples}
\newtheorem{ex}[thm]{Example}

\numberwithin{equation}{section}

\DeclareMathOperator{\KP}{KP}
\DeclareMathOperator{\spn}{span}
\DeclareMathOperator{\lsp}{span}
\DeclareMathOperator{\Soc}{Soc}
\DeclareMathOperator{\End}{End}

\newcommand{\Z}{\mathbb{Z}}
\newcommand{\N}{\mathbb{N}}

\newcommand{\inv}{^{-1}}

\title[The socle and semisimplicity of a Kumjian-Pask algebra]
{\boldmath{The socle and semisimplicity of a Kumjian-Pask algebra}}

\author[J. H. Brown]{Jonathan H. Brown}
\author[A. an Huef]{Astrid an Huef}
\address{Department of Mathematics and Statistics\\
University of Otago\\
Dunedin 9054\\
New Zealand}
\email{jbrown@maths.otago.ac.nz, astrid@maths.otago.ac.nz}
\thanks{This research was supported by the University of Otago.}

\date{9 January, 2012}

\begin{document}
\begin{abstract}
The Kumjian-Pask algebra $\KP(\Lambda)$ is a graded algebra associated to a higher-rank graph $\Lambda$ and is a generalization of the Leavitt path algebra of a directed graph. We  analyze the minimal left-ideals of $\KP(\Lambda)$, and  identify its socle as a graded ideal by describing its generators in terms of a subset of vertices of the graph. We characterize when $\KP(\Lambda)$  is semisimple, and obtain a complete structure theorem for a semisimple Kumjian-Pask algebra.
\end{abstract}

\maketitle

\section{Introduction}

The Leavitt path algebra $L(E)$ of a directed graph $E$ over a field is an algebra defined by generators and relations encoded in the graph.  Abrams and Aranda Pino \cite{AA05} constructed  Leavitt path algebras in analogy with the graph $C^*$-algebras introduced in \cite{KPRR97}.  Leavitt path algebras are so called because they generalize the algebras without invariant basis number studied by Leavitt in \cite{Lea57}.  Both Leavitt path algebras and graph $C^*$-algebras provide a rich source of examples with substantial structure theories and have attracted interest from a broad range of researchers.  Results and techniques used to study Leavitt path algebras are now applied to study the $C^*$-algebras of directed graphs (for example \cite{AMP07, AMO09})  and vice-versa (for example, \cite{AA05, AA06, Tom07}).  In this paper we study analogues of  Leavitt path algebras associated to higher-rank graphs; these algebras are called  Kumjian-Pask algebras. 

A higher-rank graph  $\Lambda$ is a combinatorial object in which ``paths'' have a $k$-dimensional shape or degree, and a $1$-graph reduces to a directed graph. Kumjian and Pask introduced $k$-graphs and the associated $C^*$-algebras $C^*(\Lambda)$ in \cite{KumPas00} to provide  models for  higher-rank versions of  Cuntz-Krieger algebras studied by Robertson and Steger in \cite{RS}.  The $C^*$-algebras of $k$-graphs have been extensively studied (see, for example, \cite{Bur09, DPY08, ES,  FMY05, KP06, PRRS06,  RSY03, RS09, SZ08, Yee07}), but open problems remain. 
For example, Evans and Sims 
have recently found necessary conditions on $\Lambda$ for $C^*(\Lambda)$ to be 
approximately finite-dimensional, but the converse seems elusive unless the graph is finitely aligned and has finitely many vertices \cite{ES}.
Evans and Sims also proved that if $\Lambda$ is a finitely-aligned  $k$-graph with finitely many vertices, then
$C^*(\Lambda)$ is finite-dimensional if and only if $\Lambda$ contains
no ``cycles'' (and if so then $C^*(\Lambda)$ is isomorphic to a direct
sum of matrix algebras indexed by the set of sources of of the graph).

The Kumjian-Pask algebra $\KP(\Lambda)$ is an  algebraic version of the $C^*$-algebra associated to a row-finite $k$-graph with no sources.  Defined and studied in \cite{A-PCaHR11}, $\KP(\Lambda)$ has a  universal property based on a family of generators satisfying suitable Cuntz-Krieger relations,  is $\Z^k$-graded,  and its two-sided graded ideal structure is given by saturated and hereditary subsets of vertices of the $k$-graph. Example~7.1 in \cite{A-PCaHR11} shows that there are Kumjian-Pask algebras that do not arise as  Leavitt path algebras. 

In this paper we identify the minimal left-ideals of the Kumjian-Pask algebra as those associated to vertices called ``line points'',  use this analysis to identify the socle of the algebra as a graded ideal, and characterize when the algebra is semisimple.  We also obtain  a complete structure theorem for semisimple Kumjian-Pask algebras reminiscent  of the Wedderburn-Artin Theorem.

The socle of a Leavitt path algebra has been studied in \cite{ARS09, PBGM10,A-PM-BM-GS-M08}.  We follow \cite{A-PM-BM-GS-M08} in our analysis of the minimal left-ideals in a Kumjian-Pask algebra and  \cite{APPM10} to prove our structure theorem for semisimple algebras.    But our methods are often simpler than those of  \cite{APPM10} and \cite{A-PM-BM-GS-M08}, even when $k=1$. The proofs in \cite{APPM10} and \cite{A-PM-BM-GS-M08} rely on delicate constructions  of specific cycles which we replace  with  comparison of degrees of paths.   Also, our analysis of the direct sum in a semisimple algebra goes a bit further than \cite{APPM10} because we can  write our direct sum over an explicit index set  of equivalence classes of vertices in Theorem~\ref{thm: KP ss matrix}. 
In
particular, Theorem~\ref{thm: KP ss matrix} shows first that  a finite-dimensional
Kumjian-Pask algebra cannot  arise from a  $k$-graph with no sources,
and second, that a semisimple  Kumjian-Pask algebra is isomorphic to a
Leavitt path algebra (see Remark~\ref{rmk-ss} ).


\section{Preliminaries on $k$-graphs and Kumjian-Pask algebras}\label{sec: prelim}
We start by establishing our conventions and proving some basic properties of Kumjian-Pask algebras.

Let $k$ be a positive integer. We consider the additive semigroup $\N^k$ as a category with one object.  Following \cite[Definition~1.1]{KumPas00}, we say a countable  category $\Lambda=(\Lambda^0, \Lambda, r,s)$, with objects $\Lambda^0$, morphisms $\Lambda$, range map $r$ and source map $s$, is  a \emph{$k$-graph} if there exists a functor $d:\Lambda\to \N^k$ with the \emph{unique factorization property}:   if $d(\lambda)=m+n$ for some $m,n\in \N^k$, then there exist unique $\mu,\nu\in \Lambda$ such that $r(\nu)=s(\mu)$ and $d(\mu)=m, d(\nu)=n$ with $\lambda=\mu\nu$.  
Since we think of $\Lambda$ as a generalized graph, we call $\lambda\in \Lambda$ a \emph{path} in $\Lambda$ and $v\in \Lambda^0$ a \emph{vertex}.

 Let $n\in \N^k$. We use the factorization property to identify $\Lambda^0$ and the paths of degree $0$, and then define $\Lambda^n:=d\inv(n)$ and $\Lambda^{\neq n}=\Lambda\setminus \Lambda^n$.   For $X, Y\subset \Lambda^0$  we denote by 
\[
X\Lambda:=\{\lambda\in \Lambda: r(\lambda)\in X\},\quad\Lambda Y:=\{\lambda\in \Lambda: s(\lambda)\in Y\},\quad X\Lambda Y:=X\Lambda\cap \Lambda Y,
\]
and  $X\Lambda^n:=X\Lambda\cap \Lambda^n,  \Lambda^n Y:=\Lambda Y\cap \Lambda^n,$ and $X\Lambda^n Y:=X\Lambda Y\cap \Lambda^n$.
For simplicity, we write $v\Lambda$ for $\{v\}\Lambda$.

A $k$-graph $\Lambda$ is \emph{row-finite} if $|v\Lambda^n|<\infty$ for all $v\in \Lambda^0, n\in \N^k$ and has \emph{no sources} if $v\Lambda^n\neq \emptyset$ for all $v\in\Lambda^0, n\in \N^k$.  We assume throughout that $\Lambda$ is a row-finite $k$-graph with no sources, that is $0<|v\Lambda^n|<\infty$ for all $v\in \Lambda^0, n\in \N^k$.

For $m=(m_1,\ldots, m_k),n=(n_1,\ldots, n_k)\in \N^k$ we say $m\leq n$ if $m_i\leq n_i$ for all $i$.  We denote the join of $m$ and $n$ by $m\vee n$: so $m\vee n=(\max\{m_1,n_1\},\ldots \max\{m_k,n_k\})$.

\begin{exs} These examples first appear in \cite[Example~1.3 and Examples~1.7]{KumPas00}.  We discuss them here to establish notation.
\begin{enumerate} \item Let $E=(E^0, E^1, r_E, s_E)$ be a directed graph.  Our convention is that  a path $\mu$ in $E$ of length $|\mu|$ is a concatenation $\mu=\mu_1\mu_2\cdots\mu_{|\mu|}$ of edges $\mu_i\in E^1$ with $s_E(\mu_i)=r_E(\mu_{i+1})$. (The usual convention in papers on Leavitt path algebras is to write  $\mu=\mu_1\mu_2\cdots\mu_{|\mu|}$  with $r_E(\mu_i)=s_E(\mu_{i+1})$. We deviate from the usual convention so that concatenation of paths is compatible with  composition of morphisms.)
 Take $\Lambda_E^0$ to be the set of vertices, $\Lambda_E$   the collection of finite paths, composition to be concatenation of paths, and $r=r_E$, $s=s_E$. Then  $\Lambda_E=(E^0, \Lambda, r,s)$ is a category. Equipped with the length function $d:\Lambda_E\to\N$, $\Lambda_E$ is a $1$-graph. 
Conversely, every $1$-graph $\Lambda$ gives a directed graph $E_\Lambda$ using the morphisms of degree $1$ for the edges. \item Let $\Omega_k=\{(m,n)\in\N^k\times \N^k:m\leq n\}$.  Then $\Omega_k$ is category with objects  $\N^k$, morphisms $(m,n)\in \Omega_k$,  $r(m,n)=m$ and $s(m,n)=n$, and composition $(m,n)(n,p)=(m,p)$.  Equipped with degree map $d$, defined by $d(m,n)=n-m$, $\Omega_k$ is a $k$-graph.
\end{enumerate}
\end{exs}

As in \cite[Definition~2.1]{KumPas00}, define the \emph{infinite path space} of $\Lambda$ to be 
\[
\Lambda^\infty=\{x:\Omega_k\to \Lambda: x \quad\text{is a degree-preserving functor}\}
\]
 and $v\Lambda^\infty=\{x\in \Lambda^\infty: x(0)=v\}$.  Since we assume that $\Lambda$ has no sources, $v\Lambda^\infty\neq \emptyset$ for all $v\in \Lambda^0$. As in \cite[Remark~2.2]{KumPas00}, $x=y\in \Lambda^\infty$ if and only if $x(0,n)=y(0,n)$ for all $n\in \N^k$.  For $p\in \N^k$ we define a map $\sigma^p:\Lambda^\infty\to \Lambda^\infty$ by $\sigma^p(x)(m,n)=x(m+p,n+p)$.   Note that $x=x(0,p)\sigma^p(x)$.  We say that $x\in \Lambda^\infty$ is \emph{periodic} if there exists  $p\neq q\in \N^k$ such that $\sigma^p(x)=\sigma^q(x)$: that is for all $m,n\in \N^k$ $x(m+p,n+p)=x(m+q, n+q)$; $x$ is called \emph{aperiodic} otherwise.  A $k$-graph $\Lambda$ is called \emph{aperiodic} if for every vertex $v\in \Lambda^0$ there exists an aperiodic path in $v\Lambda^\infty$.

\begin{ex} Let $\Lambda$ be a $1$-graph and suppose that $x\in \Lambda^\infty$ is periodic, that is, there exists $p\neq q\in \N$ such that $\sigma^p(x)=\sigma^q(x)$.  Since $\N$ is totally ordered, we may assume $p<q$.  Then $x(p,q)$ is a  cycle in the sense that $r(x(p,q))=s(x(p,q))$.\end{ex}

Let $\Lambda$ be a row-finite $k$-graph with no sources and $R$  a commutative ring with $1$. Let $G(\Lambda^{\neq 0}):=\{\lambda^*:\lambda\in \Lambda^{\neq 0}\}$ be a copy of the paths  called \emph{ghost paths}. We extend $r, s$ and $d$ to 
$G(\Lambda^{\neq 0})$ by $r(\lambda^*)=s(\lambda)$, $s(\lambda^*)=r(\lambda)$ and $d(\lambda^*)=-d(\lambda)$. We also extend $d$ to words $w=w_1\cdots w_{|w|}$ on $\Lambda^0\cup\Lambda^{\neq 0}\cup G(\Lambda^{\neq 0})$ by $d(w)=\sum_i d(w_i)$.

We now recall \cite[Definition~3.1]{A-PCaHR11}.
A \emph{Kumjian-Pask $\Lambda$-family} $(P,S)$ in an $R$-algebra $A$ consists of  functions $P:\Lambda^0\to A$ and $S:\Lambda^{\neq 0}\cup G(\Lambda^{\neq 0})\to A$ such that
\begin{enumerate}
\item[(KP1)]  $\{P_v:v\in \Lambda^0\}$ is a set of mutually orthogonal idempotents;  
\item[(KP2)]  for $\lambda, \mu\in \Lambda^{\neq 0}$ with $r(\mu)=s(\lambda)$ we have
\[S_{\lambda}S_\mu=S_{\lambda\mu}\quad S_{\mu^*}S_{\lambda^*}=S_{(\lambda\mu)^*}\quad P_{r(\lambda)}S_\lambda=S_\lambda=S_\lambda P_{s(\lambda)}\quad P_{s(\lambda)}S_{\lambda^*}=S_{\lambda^*}=S_{\lambda^*}P_{r(\lambda)};\]
\item[(KP3)] for all $\lambda,\mu\in \Lambda^{\neq 0}$ with $d(\lambda)=d(\mu)$ we have
$S_{\lambda^*}S_\mu=\delta_{\lambda,\mu}P_{s(\lambda)}$;
\item[(KP4)] for all $v\in \Lambda^0$ and $n\in \N^k\setminus\{0\}$ we have
$P_v=\sum_{\lambda\in v\Lambda^n} S_\lambda S_{\lambda^*}$.
\end{enumerate}

It is proved in  \cite[Theorem~3.4]{A-PCaHR11} that there is 
an $R$-algebra $\KP_R(\Lambda)$ generated by a Kumjian-Pask $\Lambda$-family $(p, s)$ with the following universal property:  whenever $(Q,T)$ is a Kumjian-Pask $\Lambda$-family in an $R$-algebra $A$, there is a
unique $R$-algebra homomorphism $\pi_{Q,T}:\KP_R(\Lambda)\to A$ such that
\begin{equation*}\label{defpiqt}
\pi_{Q,T}(p_v) = Q_v, \quad \pi_{Q,T}(s_{\lambda}) = T_{\lambda}, \quad \pi_{Q,T}(s_{\mu^*}) =T_{\mu^*}
\end{equation*}
for $v\in \Lambda^0$ and $\lambda,\mu\in\Lambda^{\neq 0}$. The generating family $(p,s)$ is called the universal Kumjian-Pask family.    It is a consequence of the \emph{Kumjian-Pask relations} (KP1)--(KP4) that $\KP_R(\Lambda)=\spn_R\{s_\mu s_\nu^*:\mu,\nu\in \Lambda \text{\ with $s(\mu)=s(\nu)$}\}$ (with the convention that if $v\in \Lambda_0$ then $s_v=p_v$).  Indeed, if $s(\mu)=s(\nu)$ then $s_\mu s_{\nu^*}\neq 0$ in $\KP_R(\Lambda)$.  For every nonzero $a\in \KP_R(\Lambda)$ and $n\in \N^k$ there exist $m\geq n$ and  a finite subset $F\subset\Lambda\times \Lambda^m$ such that $s(\alpha)=s(\beta)$ for all $(\alpha,\beta)\in F$ and 
\begin{equation}\label{normalform}
a=\sum_{(\alpha,\beta)\in F \subset\Lambda\times \Lambda^m}r_{\alpha,\beta} s_{\alpha} s_{\beta^*}\ \text{with}\ r_{\alpha,\beta}\in K\setminus\{0\};
\end{equation}
in this case we say that $a$ is written in \emph{normal form} \cite[Lemma~4.2]{A-PCaHR11}.  The condition $F\subset \Lambda\times \Lambda^m$ says that all the $\beta$ appearing in \eqref{normalform} are of the same degree $m$. It is proved in \cite[Theorem~3.4]{A-PCaHR11} that the subgroups 
\[
\KP_R(\Lambda)_n:=\spn_R\{s_\lambda s_{\mu^*}:\lambda,\mu\in \Lambda\ \text{and}\ d(\lambda)-d(\mu)=n\}\quad\quad (n\in\Z^k)
\]
give a $\Z^k$-grading of $\KP_R(\Lambda)$.

Lemma~4.3 of \cite{A-PCaHR11} says that for all nonzero $a\in \KP_R(\Lambda)$ there exist $\mu,\nu\in \Lambda$ such that $s_{\mu^*}as_\nu\neq 0$; the  proof of this lemma actually gives a stronger result, and we can then strengthen the conclusion of \cite[Proposition~4.9]{A-PCaHR11} as well.  

\begin{lem}\label{lem: strengthen} Suppose $0\neq a=\sum_{(\alpha,\beta)\in F} r_{\alpha,\beta} s_{\alpha}s_{\beta^*}\in\KP_R(\Lambda)$ is in normal form. 
\begin{enumerate}
\item\label{lem: strengthen 1}
For all $(\mu,\nu)\in F$, $$ s_{\mu^*}as_{\nu}=r_{\mu,\nu}p_{s(\mu)}+\sum_{\substack{(\alpha,\nu)\in F\\ d(\alpha)\neq d(\mu)}} r_{\alpha,\nu}s_{\mu^*}s_{\alpha}\neq 0.$$
\item\label{lem: strengthen 2}
If there exists $v\in \{s(\alpha):(\alpha,\beta)\in F\}$ such that $v\Lambda^\infty$ contains an aperiodic path, then there exist $\gamma,\eta\in \Lambda$ such that $s_{\gamma^*}as_{\eta}=rp_{s(\gamma)}\neq 0$.
\end{enumerate}
\end{lem}

\begin{proof} \eqref{lem: strengthen 1} Fix $(\mu, \nu)\in F$.  If $(\alpha,\beta)\in F$ then $d(\beta)=d(\nu)$.  If  $\mu\neq \alpha$ and $s_{\mu^*}s_{\alpha}\neq 0$ then $d(\mu)\neq d(\alpha)$ by (KP3). So another two applications of (KP3) give
\begin{equation*}
s_{\mu^*}as_{\nu}
=\sum_{(\alpha,\beta)\in F} r_{\alpha,\beta} s_{\mu^*}s_{\alpha}(s_{\beta^*}s_{\nu})
=\sum_{(\alpha,\nu)\in F} r_{\alpha,\nu} s_{\mu^*}s_{\alpha}
=r_{\mu,\nu}p_{s(\mu)}+\sum_{\substack{(\alpha,\nu)\in F\\ d(\alpha)\neq d(\mu)}} r_{\alpha,\nu}s_{\mu^*}s_{\alpha}.
\end{equation*}  
Therefore the zero-graded component of $s_{\mu^*}as_{\nu}$ is precisely $r_{\mu,\nu}p_{s(\mu)}$. Since $r_{\mu,\nu}\neq 0$, $r_{\mu,\nu}p_{s(\mu)}\neq 0$ by \cite[Theorem~3.4]{A-PCaHR11}.   Thus $s_{\mu^*}as_{\nu}\neq 0$.

\eqref{lem: strengthen 2} Pick $(\mu,\nu)\in F$ such that $v=s(\mu)$.  Then  for this $(\mu,\nu)\in F$ we have  \[0\neq s_{\mu^*}as_\nu=r_{\mu,\nu}p_{s(\mu)}+\sum_{(\alpha,\nu)\in F}r_{\alpha,\nu} s_{\mu^*}s_{\alpha}\] by \eqref{lem: strengthen 1}. 
The result now follows as in the proof of \cite[Proposition~4.9]{A-PCaHR11} after observing that  while \cite[Lemma~6.2]{HRSW} is stated for aperiodic $k$-graphs the proof only requires the existence of an aperiodic path at the vertex $v$.
\end{proof}

\begin{lem}\label{lem local units} $\KP_R(\Lambda)$ has a countable set of  local units. \end{lem}
\begin{proof}
Let $a\in \KP_R(\Lambda)$, and write $a$ in normal form $a=\sum_{(\alpha,\beta)\in F}r_{\alpha,\beta} s_{\alpha} s_{\beta^*}$. 
For a finite subset $X$ of $\Lambda^0$  define $u_X=\sum_{v\in X}p_v$.   Take $W=\{r(\alpha), r(\beta):(\alpha,\beta)\in F\}$.  Then $u_Wa=a=au_W$. Thus $\{ u_{X} : X \text{\ is a finite subset of\ }\Lambda^0\}$ is a set of local units of $\KP_R(\Lambda)$; it is countable because $\Lambda^0$ is countable.
\end{proof}

A subset $H$ of $\Lambda^0$ is \emph{hereditary} if $r(\lambda)\in H$ for some $\lambda\in\Lambda$ implies $s(\lambda)\in H$. A subset $H$ of $\Lambda^0$  is \emph{saturated} if 
$v\in \Lambda^0, n\in \N^k$ and $s(v\Lambda^n)\subset H$ implies $v\in H$.  For every subset $W\subset \Lambda^0$, there exists a smallest subset $\overline{W}$ of $\Lambda^0$ such that $\overline{W}$ is both hereditary and saturated; we call $\overline{W}$ the \emph{saturated hereditary  closure} of $W$.  

\begin{lem}\label{oldRemark} Let $W$  be a subset of $\Lambda^0$ and let $\Sigma^0(W)$ be the smallest hereditary subset of $\Lambda^0$ containing $W$.  If $v\in\overline{W}$ then $v\Lambda\Sigma^0(W)\neq \emptyset$.
\end{lem}

\begin{proof}
For $N\geq 1$ set $\Sigma^N(W)=\bigcup_{i=1}^k\{v\in\Lambda^0:s(v\Lambda^{e_i})\subset\Sigma^{N-1}(W)\}$. Since $\Sigma^0(W)$ is hereditary, $\overline{W}=\bigcup_{N\geq 1}\Sigma^N(W)$ by  \cite[Lemma~5.1]{RSY03}.

Fix $v\in\overline{W}$. Choose $N$ such that $v\in \Sigma^N(W)$.  Then  there exists $i_N\in\{1,\dots, k\}$ such that $s(v\Lambda^{e_{i_N}})\subset \Sigma^{N-1}(W)$.  In particular, there exists $\mu_N\in v\Lambda\Sigma^{N-1}(W)$. If $N=1$ we are done. If not, since $s(\mu_N)\in \Sigma^{N-1}(W)$, there exists $i_{N-1}$ such that $s(s(\mu_N)\Lambda^{e_{i_{N-1}}})\subset \Sigma^{N-1}(W)$. In particular, there exists $\mu_{N-1}\in s(\mu_N)\Lambda\Sigma^{N-2}(W)$. Now $\mu_N\mu_{N-1}\in v\Lambda\Sigma^{N-2}(W)$.  Continue to obtain $\mu:=\mu_N\mu_{N-1}\dots\mu_1 \in v\Lambda\Sigma^0(W)$.
\end{proof} 

For a saturated hereditary  subset $H\subset\Lambda^0$ define $I_H$ to be the ideal of $\KP_R(\Lambda)$ generated by $\{p_v: v\in H\}$.    For an ideal $I$ in $\KP_R(\Lambda)$ define  
$H_I=\{v:p_v\in I\}$.  Then $H_I$ is  saturated hereditary by \cite[Lemma~5.2]{A-PCaHR11}, and if $H$ is saturated hereditary then $H_{I_H}=H$  by \cite[Lemma~5.4]{A-PCaHR11}.  An ideal $I$ in $\KP_R(\Lambda)$ is \emph{graded}  if $I=\oplus_{n\in \Z^k} I\cap \KP_R(\Lambda)_n$, which occurs if and only if $I$ is generated by  homogeneous elements.  An ideal $I$ is \emph{basic} if $tp_v\in I$  and $t\in R\setminus\{0\}$ implies that $p_v\in I$.  If $R=K$ is a field  then every  ideal in $\KP_K(\Lambda)$ is basic, and hence \cite[Theorem~5.1]{A-PCaHR11} implies that the map 
$H\mapsto I_H$ is a lattice isomorphism of the saturated hereditary  subsets of $\Lambda^0$ onto  the graded ideals of $\KP_K(\Lambda)$.  Here we restrict our attention to Kumjian-Pask algebras over a field $K$.  

\begin{lem}\label{lem: ideal vert} Let  $K$ be a field. Let $W\subset \Lambda^0$ and let $J$ be the ideal generated by $\{p_v: v\in W\}$.   Then  $J$ is a graded ideal of $\KP_K(\Lambda)$, and   $J=I_{\overline{W}}$ and $H_{J}=\overline{W}$.\end{lem}

\begin{proof}  Suppose $W\subset \Lambda^0$ and $J$ is the ideal generated by $\{p_v: v\in W\}$. We have $W\subset H_J=\{v:p_v\in J\}$, and $H_J$ is saturated hereditary. Thus $\overline{W}\subset H_J$.  Since $\{p_v:v\in W\}$ is a set of homogeneous elements the ideal $J$ generated by it is graded.  By \cite[Theorem~5.1]{A-PCaHR11},  $H\mapsto I_H$ is a lattice isomorphism of the  saturated hereditary subsets of $\Lambda^0$ onto  the graded ideals of $\KP_K(\Lambda)$, and  $J=I_{H_J}$.  Now $J\subset I_{\overline{W}}\subset I_{H_{J}}=J$.  Thus  $J=I_{\overline{W}}$ and $H_{J}=\overline{W}$.
\end{proof}

A ring $R$ is \emph{nondegenerate} if $aRa=\{0\}$ implies that $a=0$.   A ring $R$ is \emph{semiprime}   if whenever $I$ is an ideal  in $R$ such that $I^2=\{0\}$, then $I=\{0\}$.     

\begin{lem}\label{lem: KP nondegenerate} Let  $K$ be a field. Then $\KP_K(\Lambda)$ is nondegenerate and semiprime.  \end{lem}

\begin{proof} A ring with local units is nondegenerate if and only if it is semiprime  by, for example, \cite[Proposition~6.2.20]{Bland11}. Since $\KP_K(\Lambda)$ has local units by Lemma~\ref{lem local units},  it suffices to show that $\KP_K(\Lambda)$ is nondegenerate.

First suppose that $a\KP_K(\Lambda)a=\{0\}$
where  $a\in \KP_K(\Lambda)_n$ is homogeneous of degree $n\in \Z^k$.  By way of contradiction, suppose $a\neq 0$ and write $a=\sum_{(\alpha,\beta)\in F} r_{(\alpha,\beta)}s_{\alpha}s_{\beta^*}$ in normal form.  Since $a\in \KP_K(\Lambda)_n$, all of the summands of $a$ are in $\KP_K(\Lambda)_n$ as well. Since the $\beta$ all have the same degree, so do the $\alpha$.  Fix $(\mu,\nu)\in F$. Then $s_{\mu^*} as_{\nu}$ is homogeneous, and it follows from Lemma~\ref{lem: strengthen}\eqref{lem: strengthen 1} that  $s_{\mu^*} as_{\nu}=r_{\mu,\nu}p_{s(\mu)}$. So $r_{\mu,\nu}p_{s(\mu)}\KP_K(\Lambda)r_{\mu,\nu}p_{s(\mu)}=s_{\mu^*} as_{\nu} \KP_K(\Lambda)s_{\mu^*} as_{\nu}\subset s_{\mu^*} \left(a \KP_K(\Lambda) a\right)s_{\nu}=\{0\}$.  But now 
\[
r_{\mu,\nu}^2p_{s(\mu)} =(r_{\mu,\nu}p_{s(\mu)})p_{s(\mu)}(r_{\mu,\nu}p_{s(\mu)})\in r_{\mu,\nu}p_{s(\mu)}\KP_K(\Lambda)r_{\mu,\nu}p_{s(\mu)}=\{0\}
\] 
implies that $r_{\mu,\nu}^2=0$. Since $K$ is a field, $r_{\mu,\nu}=0$.  But $(\mu,\nu)\in F$ was arbitrary, so $a=0$, a contradiction. Thus $a\KP_K(\Lambda)a=\{0\}$ and $a$ homogeneous implies that $a=0$.

Next, suppose $a\KP_K(\Lambda)a=\{0\}$ for some $a\in\KP_K(\Lambda)$. By way of contradiction again, suppose $a\neq 0$.  Decompose $a$ into nonzero homogeneous components  $a=\sum_{n\in G}a_n$ with $a_n\in \KP_K(\Lambda)_n$ for some finite subset $G$ of $\Z^k$.  If $b\in \KP_K(\Lambda)_m$ for some $m$, then $aba\in a\KP_K(\Lambda)a=\{0\}$. Let $n, n'\in G$ such that $n\neq n'$.  Then the degrees of  $a_n b a_n$ and $a_{n'} b a_{n'}$ are different.  Therefore $a_n b a_n$ is the unique homogeneous component of $aba$ of degree $2d(a_n)+m$, and hence must be zero.  Thus $a_n ba_n=0$ for all homogeneous elements $b\in \KP_K(\Lambda)$. Since $\KP_K(\Lambda)$ is generated by homogeneous elements we have $a_n\KP_K(\Lambda)a_n=\{0\}$.  Therefore $a_n=0$ by the previous paragraph,   a contradiction.  Hence $\KP_K(\Lambda)$ is nondegenerate.  \end{proof}

\section{Minimal left-ideals  and the socle of a Kumjian-Pask algebra}\label{sec: min left ideal}

Throughout this section $\Lambda$ is a row-finite $k$-graph with no sources and $K$ is a field. 
In analogy with the analysis of the minimal left-ideals of the Leavitt path algebras in \cite{A-PM-BM-GS-M08},  we show that  minimal left-ideals of the Kumjian-Pask algebra are  isomorphic to $\KP_K(\Lambda)p_v$ where $v$ is a vertex called a ``line point'' in Definition~\ref{defn: line pt} below.  We then identify the socle of $\KP_K(\Lambda)$ as the graded ideal generated by the vertex idempotents $p_v$ where $v$ is in the saturated hereditary  closure of the line points. 

\begin{lem}\label{lem: min in corner}Let $a\in \KP_K(\Lambda)$, and suppose that $\KP_K(\Lambda)a$ is a minimal left-ideal.  Then there exist a vertex $v\in \Lambda^0$ and $b\in p_v \KP_K(\Lambda)p_v$ such that $\KP_K(\Lambda)a$ and $\KP_K(\Lambda)b$ are isomorphic as left modules.\end{lem}

\begin{proof}Suppose that $\KP_K(\Lambda)a$ is a minimal left-ideal. Then $a\neq 0$ and we write $a=\sum_{(\alpha,\beta)\in F} r_{\alpha,\beta}s_{\alpha} s_{\beta^*}$ in normal form.  Pick $(\mu,\nu)\in F$ and let $v=s(\mu)=s(\nu)$. Then, by Lemma~\ref{lem: strengthen}\eqref{lem: strengthen 1},  $0\neq s_{\mu^*}as_\nu\in p_v\KP_K(\Lambda)p_v$.  Since $\KP_K(\Lambda)a$ is minimal and $s_{\mu^*}a\neq 0$, $\KP_K(\Lambda)a =\KP_K(\Lambda)s_{\mu^*}a$.   Let $b=s_{\mu^*}as_\nu$, then $cs_{\mu^*}a\mapsto cb$  defines a left-module homomorphism, $\Phi:\KP_K(\Lambda)a\to  \KP_K(\Lambda)b$. Indeed, if $cs_{\mu^*}a=ds_{\mu^*}a$ then $(c-d)s_{\mu^*}a=0$ and hence $cb-db=(c-d) s_{\mu^*}as_\nu =0$, so $\Phi$ is well-defined. Now $\Phi$ is a surjective module homomorphism, and it is injective since the kernel of $\Phi$ is a left ideal in the minimal left-ideal $\KP_K(\Lambda)a$.  \end{proof}

By Lemma~\ref{lem: min in corner} we may  assume that minimal left-ideals are of the form $\KP_K(\Lambda)b$ for some $b\in p_v\KP_K(\Lambda)p_v$ and $v\in\Lambda^0$; next we analyze   the properties of  such $v$.

\begin{defn}\label{defn: line pt}Suppose that $\Lambda$ is a $k$-graph.  A vertex $v\in \Lambda^0$ is  a \emph{line point} if $v\Lambda^\infty=\{x\}$ and $x$ is aperiodic. We denote the set of line points of $\Lambda$ by $P_l(\Lambda)$.\end{defn}

The following observations are needed below. In particular, they help in Remark~\ref{reconcile} where we reconcile Definition~\ref{defn: line pt} with the line points of directed graphs as defined in \cite[Definition~2.1]{A-PM-BM-GS-M08}.

\begin{rmk}\label{many remarks} As pointed out in \cite[Remark~2.2]{KumPas00}, the values of $x(0,m)$ for $m\in\N^k$ completely determine $x\in \Lambda^\infty$. Thus: 
\begin{enumerate} 
\item\label{many remarks one}  $|v\Lambda^\infty|=1$ if and only if $|v\Lambda^m|=1$ for all $m\in \N^k$. 
\item\label{many remarks two}  Suppose that  $v\Lambda^\infty=\{x\}$.  Then $x(m)\Lambda^\infty=\{\sigma^m(x)\}$ for all $m\in \N^k$. Thus, if $x(p)=x(q)$ for some $p\neq q\in \N^k$ then $\sigma^p(x)=\sigma^q(x)$, that is, $x$ is periodic. 
\item\label{many remarks three} The set $P_l(\Lambda)$ is hereditary. To see this, let $\mu\in v\Lambda$ for some line point $v$.  Then $v\Lambda^\infty=\{x\}$ where $x$ is aperiodic.  By \eqref{many remarks two}, $s(\mu)\Lambda^\infty=\{\sigma^{d(\mu)}(x)\}$, and $\sigma^{d(\mu)}(x)$ is aperiodic because $x$ is.  Thus $s(\mu)$ is a line point, and $P_l(\Lambda)$ is hereditary. 
\item\label{many remarks four} If $|v\Lambda^\infty|=1$ and $\alpha\in v\Lambda$, then from \eqref{many remarks one}, $v\Lambda^{d(\alpha)}=\{\alpha\}$.  So (KP4) implies that \[p_v=\sum_{\lambda\in v\Lambda^{d(\alpha)}} s_{\lambda}s_{\lambda^*}=s_{\alpha}s_{\alpha^*}.\]
\end{enumerate}
\end{rmk}

\begin{rmk}\label{reconcile}
If $\Lambda$ is a $1$-graph then Definition~\ref{defn: line pt} reduces to \cite[Definition~2.1]{A-PM-BM-GS-M08}. To see this, let $T(v)$ be the smallest hereditary subset of $\Lambda^0$ containing $v\in\Lambda^0$.  Notice that $|v\Lambda^\infty|>1$ if and only if there exists $u\in T(v)$ with $|u\Lambda^m|>1$ for some $m\in\N$.  Now let $v$ be a line point, that is,  $v\Lambda^\infty=\{x\}$ where $x$ is aperiodic. Since  $|v\Lambda^\infty| = 1$  there can be no ``bifurcations'' at any $u\in T(v)$, and since $x$ is aperiodic  there can be no cycles  at any $u\in T(v)$. Thus $v$ is a line point in the sense of \cite[Definition~2.1]{A-PM-BM-GS-M08}.
\end{rmk}

Let $a\in p_v\KP_K(\Lambda) p_v$ and suppose that $\KP_K(\Lambda)a$ is a minimal left-ideal.  Our next goal is to show that  $v$ is a line point, and we do this in two steps: the next proposition shows that $v\Lambda^\infty$ is a singleton $\{y\}$ and Proposition~\ref{prop: min non per 2} shows that $y$ must be aperiodic. Our analysis of the minimal left-ideals culminates in Theorem~\ref{thm: min come from line}.

\begin{prop}\label{new prop pv min vL=1} Let $v\in \Lambda^0$ and $a\in p_v\KP_K(\Lambda)p_v$.
\begin{enumerate}
\item\label{new prop pv min vL=1one} Let $m\in \N^k$. Then 
\[\Psi:\KP_K(\Lambda)a\to\bigoplus_{\lambda\in v\Lambda^m} \KP_K(\Lambda)s_{\lambda}s_{\lambda^*}a, \quad\Psi(ca)=(cs_{\lambda}s_{\lambda^*}a)_{\lambda\in v\Lambda^m}\] is an isomorphism of left-modules.
\item\label{new prop pv min vL=1two} 
 If $\KP_K(\Lambda)a$ is a minimal left-ideal then $|v\Lambda^\infty|=1$.
 \end{enumerate}
\end{prop}

\begin{proof} \eqref{new prop pv min vL=1one}
Since $\Psi$ is a left-module homomorphism  it suffices to write down an inverse for it. Define $\Phi:(c_\lambda s_{\lambda}s_{\lambda^*}a)_{\lambda\in v\Lambda^m}\mapsto \sum_{_{\lambda\in v\Lambda^m}}c_\lambda s_{\lambda}s_{\lambda^*}a$.  Then $\Phi\circ\Psi(ca)=\Phi((cs_{\lambda}s_{\lambda^*}a))=\sum_{\lambda\in v\Lambda^m} cs_{\lambda}s_{\lambda^*}a=cp_va=ca$ using (KP4).  Also, 
\begin{align*}
\Psi\circ\Phi((c_\lambda s_{\lambda}s_{\lambda^*}a)_{\lambda\in v\Lambda^m})
&=\Psi\Big( \sum_{\lambda\in v\Lambda^m}c_\lambda s_\lambda s_{\lambda^*}a\Big)
=\Psi\Big( \Big(\sum_{\lambda\in v\Lambda^m}c_\lambda s_\lambda s_{\lambda^*}\Big) a\Big)\\
&=\Big(\sum_{\lambda\in v\Lambda^m}c_\lambda s_\lambda s_{\lambda^*}s_\mu s_{\mu^*}a  \Big)_{\mu\in v\Lambda^m}=(c_\lambda s_{\lambda}s_{\lambda^*}a)_{\lambda\in v\Lambda^m}
\end{align*}
using (KP3).  Thus $\Phi$ is an inverse for $\Psi$, and $\Psi$ is an isomorphism.

\eqref{new prop pv min vL=1two} Suppose that  $|v\Lambda^\infty|>1$. Then there exists $m\in\N^k$ such that  $|v\Lambda^m|>1$ (see Remark~\ref{many remarks}\eqref{many remarks one}).  Then $\KP_K(\Lambda)a$ and $\bigoplus_{\lambda \in v\Lambda^m} \KP_K(\Lambda)s_{\lambda}s_{\lambda^*}a$ are isomorphic by \eqref{new prop pv min vL=1one}.  But the direct sum has at least two summands, and hence $\KP_K(\Lambda)a$ is not minimal.
\end{proof}

The next lemma is needed again in \S\ref{sec: ss}.

\begin{lem}\label{lem: nonzero}Let $v\in \Lambda^0$ such that $|v\Lambda^\infty|=1$ and $N\in \N$. Suppose   $\mu_i,\nu_i\in v\Lambda$ such that $s(\mu_i)=s(\nu_i)$ for $i=1,\ldots ,N$.  Then $s_{\mu_1}s_{\nu_1^*}\cdots s_{\mu_N}s_{\nu_N^*}\neq 0$.\end{lem}

\begin{proof}We proceed by induction on $N$.  Since $s(\mu_1)=s(\nu_1)$, $0\neq p_{s(\mu_1)}=s_{\mu_1^*}s_{\mu_1}s_{\nu_1^*}s_{\nu_1}$.  In particular $s_{\mu_1}s_{\nu_1^*}\neq 0$, and the lemma holds for $N=1$. 

Let $N\geq 2$, and suppose  every set $\{\alpha_i,\beta_i\in v\Lambda: s(\alpha_i)=s(\beta_i), i=1,\ldots N-1\}$ has the property that $s_{\alpha_1}s_{\beta_1^*}\cdots s_{\alpha_{N-1}}s_{\beta_{N-1}^*}\neq 0$. 
Pick
$\{\mu_i,\nu_i\in v\Lambda: s(\mu_i)=s(\nu_i), i=1,\ldots, N\}$ and consider $s_{\mu_1}s_{\nu_1^*}\cdots s_{\mu_N}s_{\nu_N^*}$.
Since  $r(\nu_1)=r(\mu_2)=v$  there exist $\gamma,\eta\in \Lambda$ such that $\nu_1\gamma=\mu_2\eta\in v\Lambda^{d(\nu_1)\vee d(\mu_2)}$. Since $|v\Lambda^\infty|=1$ we have $s(\nu_1)\Lambda^{d(\gamma)}=\{\gamma\}$ and $s(\mu_2)\Lambda^{d(\eta)}=\{\eta\}$ (see Remark~\ref{many remarks}\eqref{many remarks one}).  Thus
\begin{align}
 s_{\mu_1}s_{\nu_1^*}&s_{\mu_2}s_{\nu_2^*}\cdots s_{\mu_N}s_{\nu_N^*}=s_{\mu_1}p_{s(\nu_1)}s_{\nu_1^*}s_{\mu_2}p_{s(\mu_2)}s_{\nu_2^*}\cdots s_{\mu_N}s_{\nu_N^*}\quad\text{(using (KP2))}\notag\\
&=s_{\mu_1}(s_{\gamma}s_{\gamma^*})s_{\nu_1^*}s_{\mu_2}(s_{\eta}s_{\eta^*})s_{\nu_2^*}\cdots s_{\mu_N}s_{\nu_N^*}\quad\text{(using (KP4) and Remark~\ref{many remarks}\eqref{many remarks four})}\notag\\
&=s_{\mu_1\gamma}(s_{(\nu_1\gamma)^*}s_{\mu_2\eta})s_{(\nu_2\eta)^*}\cdots s_{\mu_N}s_{\nu_N^*}\notag\\
&=s_{\mu_1\gamma}s_{(\nu_2\eta)^*}\cdots s_{\mu_N}s_{\nu_N^*} \quad\text{(by (KP3) since $\nu_1\gamma=\mu_2\eta$.)}\label{nonzerobyIH}
\end{align} 
But \eqref{nonzerobyIH}  is nonzero by our inductive hypothesis, and the lemma follows.\end{proof}

\begin{prop}\label{prop: min non per 2}Let $v\in \Lambda^0$ such that $v\Lambda^\infty=\{y\}$ and let $a\in p_v\KP_K(\Lambda)p_v$.  If $\KP_K(\Lambda)a$ is a minimal left-ideal then $y$ is aperiodic.\end{prop}

\begin{proof} Suppose that $y$ is periodic.  Then there exists $n\neq m\in \N^k$ with $\sigma^m(y)=\sigma^n(y)$: in particular, $y(m)=y(n)$.  Let $\pi_h:\Z^k\to \Z$  be the projection onto the $h$-th factor.   By possibly  interchanging the roles of $m$ and $n$ we can assume that $\pi_l(m-n)>0$ for some $l\in\{1,\ldots,k\}$. Let $\mu:=y(0,m)$ and $\nu:=y(0,n)$. Since $s(\mu)=s(\nu)$, $s_\mu s_{\nu^*}\neq 0$. Consider $s_\mu s_{\nu^*}a+a\in \KP_K(\Lambda)a$.  We have $\{0\}\neq\KP_K(\Lambda)(s_\mu s_{\nu^*}a+a)\subset \KP_K(\Lambda)a$.  We will show that this inclusion is proper by showing that $a\notin \KP_K(\Lambda)(s_\mu s_{\nu^*}a+a)$.

Let $a=\sum_{(\alpha,\beta)\in F}r_{\alpha,\beta}s_\alpha s_{\beta^*}$ be in normal form.  Then for all $(\alpha,\beta),(\mu,\nu)\in F$  we have $d(\beta)=d(\nu)$, and, since $|v\Lambda^\infty|=1$, we have $\beta=\nu$.  So we can assume that 
\[a=\sum_{j=1}^N r_j s_{\alpha_j}s_{\beta^*}\quad r_j\in K\setminus \{0\}\]
for some fixed $\beta$.  We may  assume that the $\alpha_j$ are ordered so that $\pi_l(d(\alpha_j))\leq \pi_l(d(\alpha_{j'}))$ for all $1\leq j\leq j'\leq N$.

By way of contradiction, suppose that $a\in \KP_K(\Lambda)(s_\mu s_{\nu^*}a+a)$, that is, suppose that there exists $b\in \KP_K(\Lambda)$ such that
\begin{equation}\label{an element =a}b(s_\mu s_{\nu^*}a+a)=a.\end{equation} Then $p_vbp_v(s_\mu s_{\nu^*}a+a)=p_va=a$ so we may replace $b$ by $p_vbp_v$, that is, we may assume $b\in p_v\KP_K(\Lambda)p_v$.  As above there is a $\tau\in \Lambda$ such that  $b$ has normal form 
\[b=\sum_{i=1}^M t_{i} s_{\gamma_i} s_{\tau^*}\quad\text{with $\pi_l(d(\gamma_i))\leq \pi_l(d(\gamma_{i'}))$ for $1\leq i\leq i'\leq M$.}
\]   

Our strategy is to compare degrees of the  terms on the left and right side of equation~\eqref{an element =a}; in particular we will compare the terms whose degree has  maximal and minimal $l$-th component.  The terms of $b(s_\mu s_{\nu^*}a+a)$ are either of the form $s_{\gamma_i} s_{\tau^*}s_{\alpha_j}s_{\beta^*}$ (the summands of $ba$) or  $s_{\gamma_i} s_{\tau^*}s_{\mu}s_{\nu^*}s_{\alpha_j}s_{\beta^*}$ (the summands of $bs_\mu s_{\nu^*}a$).  Note that both $s_{\gamma_i} s_{\tau^*}s_{\alpha_j}s_{\beta^*}$ and $s_{\gamma_i} s_{\tau^*}s_{\mu}s_{\nu^*}s_{\alpha_j}s_{\beta^*}$  are nonzero by Lemma~\ref{lem: nonzero}. 
 
For a homogeneous element $c\in \KP_K(\Lambda)_n$ of degree $n$ we will call $\pi_l(d(c))$ the $\pi_l$-degree of $c$.  
Since $d(\mu)=m, d(\nu)=n$ and $\pi_l(m-n)>0$, the terms of $b(s_\mu s_{\nu^*}a+a)$ who have  maximal $\pi_l$-degree come from $bs_\mu s_{\nu^*}a$. Consider the term of maximal $\pi_l$-degree, $r_N t_M s_{\gamma_M}s_{\tau^*}s_{\mu}s_{\nu^*}s_{\alpha_N}s_{\beta^*}$.  Then
$$\pi_l(d(\gamma_M\tau^*\mu\nu^*\alpha_N\beta^*))=\pi_l(d(\gamma_M\tau^*))+\pi_l(m-n)+\pi_l(d(\alpha_N\beta^*)).$$
By assumption on the ordering of $\sum r_i s_{\alpha_i}s_{\beta^*}$, the maximal $\pi_l$-degree of terms on the right-hand side of equation~\eqref{an element =a} is $\pi_l(d(\alpha_N\beta^*))$.  So for \eqref{an element =a} to hold we must have $\pi_l(d(\gamma_M\tau^*))+\pi_l(m-n)+\pi_l(d(\alpha_N\beta^*))= \pi_l(d(\alpha_N\beta^*))$. Hence \[\pi_l(d(\gamma_M\tau^*))=-\pi_l(m-n)<0.\]  But  $\pi_l(d(\gamma_M\tau^*))$ is maximal among the $\pi_l$-degrees of $b$,  so $\pi_l(d(\gamma_i\tau^*))<0$ for all $i$.  

Since $\pi_l(m-n)>0$, the terms with minimal $\pi_l$-degree in $b(s_\mu s_{\nu^*}a+a)$ come from $ba$.   Consider the term of minimal $\pi_l$-degree, $r_1t_1 s_{\gamma_1}s_{\tau^*}s_{\alpha_1}s_{\beta^*}$.  From above, $\pi_l(d(\gamma_1\tau^*))<0$, so 
$\pi_l(d(\gamma_1\tau^*\alpha_1\beta^*))=\pi_l(d( \gamma_1\tau^*))+\pi_l(d( \alpha_1\beta^*))<\pi_l(d( \alpha_1\beta^*))$.  But $s_{\alpha_1}s_{\beta^*}$ is the term on the right-hand side of \eqref{an element =a} of minimal $\pi_l$-degree.  So there is no component of $a$ that has degree $d(\gamma_1\tau^*\alpha_1\beta^*)$ and hence \eqref{an element =a} cannot hold. 
So there does not exist $b\in \KP_K(\Lambda)$ such that $b(s_\mu s_{\nu^*}a+a)=a$. Hence $\{0\}\neq\KP_K(\Lambda)(as_\mu s_{\nu^*}+a)\subsetneq \KP_K(\Lambda)a$; in particular, $\KP_K(\Lambda)a$ is not a minimal left-ideal. \end{proof}

\begin{lem}\label{lem: line gives K}If $v\in P_l(\Lambda)$ then $p_v\KP_K(\Lambda)p_v=Kp_v$.\end{lem}

\begin{proof} Let $v\in P_l(\Lambda)$, that is, $v\Lambda^\infty=\{x\}$ where $x$ is aperiodic.  To see  $p_v\KP_K(\Lambda)p_v\subset Kp_v$, let $p_vs_\alpha s_{\beta^*}p_v$ be nonzero. Then $s(\alpha)=s(\beta)$ and $\alpha,\beta\in v\Lambda$. It follows from Remark~\ref{many remarks}\eqref{many remarks two} that $\alpha=\beta$.  Now $
p_vs_\alpha s_{\beta^*}p_v=p_vs_\alpha s_{\alpha^*}p_v=p_v
$
by (KP4) and  Remark~\ref{many remarks}\eqref{many remarks four}.

Since elements of the form $s_\alpha s_{\beta^*}$ generate $\KP_K(\Lambda)$ we get $p_v\KP_K(\Lambda)p_v\subset Kp_v$.  The other inclusion is immediate because $p_v=p_v^3$.
\end{proof}

\begin{thm}\label{thm: min come from line}  Let $\Lambda$ be a row-finite $k$-graph  with no sources and let $K$ be a field.   
\begin{enumerate}
\item\label{thm: min come from line one} Let $v\in P_l(\Lambda)$. Then $\KP_K(\Lambda)p_v$ is a minimal left-ideal if and only if $v\in P_l(\Lambda)$.
\item\label{thm: min come from line two}  Let $a\in \KP_K(\Lambda)$. Then $\KP_K(\Lambda)a$ is a minimal left-ideal if and only if there exists $v\in P_l(\Lambda)$ such that $\KP_K(\Lambda)p_v$ is isomorphic to $\KP_K(\Lambda)a$ as left $\KP_K(\Lambda)$-modules.
\end{enumerate}
\end{thm}

\begin{proof}
\eqref{thm: min come from line one}
Let $v\in P_l(\Lambda)$.   We will show that $\KP_K(\Lambda)ap_v=\KP_K(\Lambda)p_v$ for all nonzero $a\in \KP_K(\Lambda)$; it then follows that $\KP_K(p_v)$ is a minimal left-ideal.
 Fix $0\neq a\in \KP_K(\Lambda)$. It is enough to show $p_v\in \KP_K(\Lambda)ap_v$.  By Lemma~\ref{lem: KP nondegenerate}, $ap_v\KP_K(\Lambda)ap_v\neq \{0\}$.  Thus there exists $b\in \KP_K(\Lambda)$ such that $p_vbap_v\neq 0$.  By  Lemma~\ref{lem: line gives K}, $p_v\KP_K(\Lambda)p_v=Kp_v$ so  there exists $t\in K\setminus\{0\}$ such that $tp_v=p_vbap_v$.  Now $p_v=(t\inv p_vb)ap_v\in \KP_K(\Lambda)ap_v$, giving $\KP_K(\Lambda)ap_v=\KP_K(\Lambda)p_v$. Thus $\KP_K(\Lambda)p_v$ is minimal.

Conversely, suppose that $\KP_K(\Lambda)p_v$ is minimal for some $v\in\Lambda^0$. By item \eqref{new prop pv min vL=1two} of  Proposition~\ref{new prop pv min vL=1} we have $v\Lambda^\infty=\{y\}$ for some $y$. Now  Proposition~\ref{prop: min non per 2} implies that $y$ is aperiodic. Thus $v\in P_l(\Lambda)$.

\eqref{thm: min come from line two}
Suppose $\KP_K(\Lambda)a$ is minimal.  By Lemma~\ref{lem: min in corner}, we can assume $a\in p_w\KP_K(\Lambda)p_w$ for some $w\in \Lambda^0$. Since $\KP_K(\Lambda)a$ is minimal, $w\Lambda^\infty=\{x\}$  by Proposition~\ref{new prop pv min vL=1}\eqref{new prop pv min vL=1two}, and then $x$ is aperiodic by Proposition~\ref{prop: min non per 2}.  Thus $w\in P_l(\Lambda)$.  Since $x\in w\Lambda^\infty$ is aperiodic,  Lemma~\ref{lem: strengthen}\eqref{lem: strengthen 2} shows there exist $\gamma,\eta\in \Lambda$ such that $s_{\gamma^*}as_{\eta}=rp_{s(\gamma)}\neq 0$ for some $r\in K$.  But $a\in p_w\KP_K(\Lambda)p_w$ so $\gamma,\eta\in w\Lambda$.  Since $\KP_K(\Lambda)a$ is minimal,  $\KP_K(\Lambda)a=\KP_K(\Lambda)r\inv s_{\gamma^*}a$. 

We claim that the map $cs_{\gamma^*}a\mapsto cs_{\gamma^*}as_{\eta}$ defines an isomorphism from 
$\KP_K(\Lambda)r\inv s_{\gamma^*}a$ to $\KP_K(\Lambda)s_{\gamma^*}as_\eta$.  This map is surjective  by definition, and is injective because the kernel is a left-ideal in the minimal left-ideal $\KP_K(\Lambda)s_{\gamma^*}a$, hence must be $\{0\}$.  Thus $\KP_K(\Lambda)a=\KP_K(\Lambda)r\inv s_{\gamma^*}a$ is isomorphic to $ \KP_K(\Lambda)s_{\gamma^*}as_\eta=\KP_K(\Lambda)p_{s(\gamma)}$ as claimed. 

Finally, $r(\gamma)=w\in P_l(\Lambda)$ and $P_l(\Lambda)$ is hereditary by Remark~\ref{many remarks}\eqref{many remarks four}, and hence $s(\gamma)\in P_l(\Lambda)$. Thus $\KP_K(\Lambda)a$ is isomorphic to $\KP_K(\Lambda)p_{s(\gamma)}$  and  $s(\gamma)\in P_l(\Lambda)$. 
The converse is immediate from \eqref{thm: min come from line one}.\end{proof}

The \emph{left socle} $\Soc_l(A)$ of an algebra $A$ is the sum of all of its minimal left-ideals; if the set of minimal left-ideals is empty then the left socle is zero. When $A$ is semiprime, the left socle and the right socle coincide \cite[Corollary~4.3.4]{BMM96}, and then the \emph{socle} $\Soc(A):=\Soc_l(A)$ of $A$ is a two-sided ideal.

By Lemma~\ref{lem: KP nondegenerate}, the Kumjian-Pask algebra $\KP_K(\Lambda)$ is semiprime. By Theorem~\ref{thm: min come from line}, we have $\sum_{v\in P_l(\Lambda)}\KP_K(\Lambda)p_v\subset \Soc(\KP_K(\Lambda))$ where $P_l(\Lambda)$ is the set of line points. The example in the proof of \cite[Proposition~4.1]{A-PM-BM-GS-M08} shows the reverse containment does not hold in general. In Theorem~\ref{thm: soc cor line} we show that $\Soc(\KP_K(\Lambda))$ is the ideal corresponding to the saturated hereditary  closure of the line points.  

When $k=1$ the first item of Theorem~\ref{thm: soc cor line} below is \cite[Theorem~4.2]{A-PM-BM-GS-M08} and the second item is \cite[Theorem~4.3]{S}. An ideal $I$ in an algebra is \emph{essential} if $a\in A$ and $aI=\{0\}$ implies $a=0$.

\begin{thm}\label{thm: soc cor line} Let $\Lambda$ be a row-finite $k$-graph  with no sources and let $K$ be a field. 
\begin{enumerate}
\item\label{thm: soc cor line one}   Let $J$ be the ideal generated by $\{p_v:v\in P_l(\Lambda)\}$. Then \[\Soc(\KP_K(\Lambda))=J=I_{\overline{P_l(\Lambda)}}.\]
\item\label{thm: soc cor line two} $\Soc(\KP_K(\Lambda))$ is an essential ideal of $\KP_K(\Lambda)$ if and only if every vertex connects to a line point.
\end{enumerate} \end{thm}

\begin{proof} \eqref{thm: soc cor line one}  The proof is very similar to that of \cite[Theorem~4.2]{A-PM-BM-GS-M08}. By Lemma~\ref{lem: ideal vert}, $I_{\overline{P_l(\Lambda)}}=J$. To see  that $\Soc(\KP_K(\Lambda))\subset J$,  let $L$ be a minimal left-ideal of $\KP_K(\Lambda)$; we will show that $L\subset J$. Let $a\in L\setminus\{0\}$. Then $\KP_K(\Lambda)a\subset L$.  By the nondegeneracy of $\KP_K(\Lambda)$ (Lemma~\ref{lem: KP nondegenerate}) we have  $\KP_K(\Lambda)a\neq 0$.  Thus $L=\KP_K(\Lambda)a$ by the minimality of $L$.   By  Theorem~\ref{thm: min come from line}, $L=\KP_K(\Lambda)a$ is isomorphic to $\KP_K(\Lambda)p_v$ for some line point $v$.   Let $\phi:\KP_K(\Lambda)a\to  \KP_K(\Lambda)p_v$ be an isomorphism.  Then $\phi(a)=bp_v$   for some $b\in \KP_K(\Lambda).$  Note that since $v$ is a line point, $p_v\in J$ and so $a=\phi\inv(\phi(a))=\phi\inv(bp_v)=\phi\inv(bp_vp_v)=bp_v\phi\inv(p_v)\in J.$  Thus $L\subset J$ and hence $\Soc(\KP_K(\Lambda))\subset J$.

For the reverse containment, if $p_v\in P_l(\Lambda)$ then $\KP_K(\Lambda)p_v\subset\Soc(\KP_K(\Lambda))$.  Thus $p_v=p_vp_v\in \Soc(\KP_K(\Lambda))$ for all $v\in P_l(\Lambda)$.  Therefore $\Soc(\KP_K(\Lambda))$ is an ideal containing $\{p_v: v\in P_l(\Lambda)\}$ and so it contains $J$.  

\eqref{thm: soc cor line two}
 First suppose that every vertex in $\KP_K(\Lambda)$ connects to a line point. Let $a\in \KP_K(\Lambda)\setminus \{0\}$; we will show that $a\Soc(\KP_K(\Lambda))=aI_{\overline{P_l(\Lambda)}}\neq \{0\}$.  First we claim that $\Lambda$ is aperiodic.  To see this, let $u\in \Lambda^0$.  Then there exists a path $\lambda\in u\Lambda v$ where $v$ is a line point. But  $v\Lambda^\infty=\{x\}$ with $x$ aperiodic, and now $\lambda x$ is an aperiodic path in $u\Lambda^\infty$.  Thus $\Lambda$ is aperiodic as claimed.  
 
 Since $\Lambda$ is aperiodic there exist $\mu,\nu\in \Lambda$ such that $s_{\mu^*} as_\nu=tp_w$ for some $w\in \Lambda^0$ and some $t\in K\setminus \{0\}$ (see Lemma~\ref{lem: strengthen}\eqref{lem: strengthen 2}).  Since $w$ connects to a line point, there exists $\gamma\in w\Lambda P_l(\Lambda)$.  Thus $p_{s(\gamma)}\in  I_{\overline{P_l(\Lambda)}}$ and then $s_\nu s_\gamma=s_\nu s_\gamma p_{s(\gamma)}\in I_{\overline{P_l(\Lambda)}}$.  Therefore
 \[0\neq p_{s(\gamma)}=s_{\gamma^*}s_\gamma=s_{\gamma^*}p_ws_\gamma=t\inv s_{\gamma^*}s_{\mu^*} as_\nu s_\gamma\in t\inv s_{\gamma^*}s_{\mu^*} a I_{\overline{P_l(\Lambda)}}.
 \] 
 Thus  $aI_{\overline{P_l(\Lambda)}}\neq \{0\}$ as well.

Conversely, suppose that $\Soc(\KP_K(\Lambda))=I_{\overline{P_l(\Lambda)}}$ is essential. Pick $w\in \Lambda^0$. Since $I_{\overline{P_l(\Lambda)}}$ is essential, there exists $b\in I_{\overline{P_l(\Lambda)}}$ such that $p_w b\neq 0$.  By \cite[Lemma~5.4]{A-PCaHR11}, 
$I_{\overline{P_l(\Lambda)}}=\lsp\{s_\alpha s_{\beta^*}: s(\alpha)=s(\beta)\in \overline{P_l(\Lambda)}\}$ 
so  we can write $b=\sum_{(\alpha,\beta)\in F} r_{\alpha,\beta}s_{\alpha}s_{\beta^*}$ with $s(\alpha)=s(\beta)\in \overline{P_l(\Lambda)}$.  Since $p_w b\neq 0$, there exists $(\mu,\nu)\in F$ such that $p_ws_{\mu}s_{\nu^*}\neq 0$.  For this $\mu$ we  have  $\mu\in w\Lambda \overline{P_l(\Lambda)}$.  Since every element in $\overline{P_l(\Lambda)}$ connects to $P_l(\Lambda)$ by Lemma~\ref{oldRemark}, we can  pick $\eta\in s(\mu)\Lambda P_l(\Lambda)$.  Thus $\mu\eta$ connects $w$ to $P_l(\Lambda)$.
\end{proof}

The next corollary follows immediately from Theorem~\ref{thm: soc cor line}\eqref{thm: soc cor line one} and  the lattice isomorphism $H\mapsto I_H$ of the saturated hereditary  subsets of $\Lambda^0$ onto the graded ideals of $\KP_K(\Lambda)$ from \cite[Theorem~5.1]{A-PCaHR11}.

\begin{cor}\label{cor: 0 soc plus}  Let $\Lambda$ be a row-finite $k$-graph  with no sources and let $K$ be a field. 
\begin{enumerate}
\item $\Soc(\KP_K(\Lambda))$ is a graded ideal of $\KP_K(\Lambda)$.
\item 
 $\Soc(\KP_K(\Lambda))=\{0\}$ if and only if $P_l(\Lambda)=\emptyset$.
 \item  $\Soc(\KP_K(\Lambda))=\KP_K(\Lambda)$ if and only if $\overline{P_l(\Lambda)}=\Lambda^0$.
\end{enumerate}
\end{cor}
 
 \begin{exs} \begin{enumerate} \item Let $\Lambda$ be a row-finite $k$-graph with one vertex $v$.  Then $P_l(\Lambda)=\emptyset$ and hence $\Soc(\KP_K(\Lambda))=\{0\}$ (for if there is just one infinite path then it must be periodic). In particular, the Kumjian-Pask algebras of the $2$-graphs with one vertex, $N_1$ blue edges and $N_2$ red edges considered by Davidson and Yang in \cite{DY09} all have zero socle.
 
\item  Every vertex of $\Omega_k$ is a line point, and hence $\Soc(\KP_K(\Omega_k))=\KP_K(\Omega_k)$.
\end{enumerate}
 \end{exs}

\section{Semisimple Kumjian-Pask algebras}\label{sec: ss}
Throughout this section $\Lambda$ is a row-finite $k$-graph with no sources and $K$ is a field. 
A semiprime ring $R$ is \emph{semisimple} if $\Soc(R)=R$. By \cite[Remark~4.3.5]{BMM96},  $\Soc(R)$ is a direct sum of minimal left-ideals. Thus $R$ is semisimple if and only if it is a direct sum of its minimal left-ideals.   By Corollary~\ref{cor: 0 soc plus}, $\KP_K(\Lambda)$ is semisimple if and only if the saturated hereditary  closure of the line points  is all of $\Lambda^0$.  In this section we study semisimple Kumjian-Pask algebras. The proof of the next lemma is very similar to the proof of its Leavitt path algebra version \cite[Lemma~2.2]{APPM10}.

\begin{lem}\label{lem: div ring}  Let $e$ be an idempotent in  $\KP_K(\Lambda)$.  Then the corner $e\KP_K(\Lambda)e$ is a division ring if and only if it is isomorphic to $K$ as a ring.\end{lem}

\begin{proof}Suppose $e\KP_K(\Lambda)e$ is a division ring. By Lemma~\ref{lem: KP nondegenerate}, $\KP_K(\Lambda)$ is  semiprime.  Now \cite[Proposition~4.3.3]{BMM96} applies, giving that $\KP_K(\Lambda)e$ is a minimal left-ideal.  By Theorem~\ref{thm: min come from line}\eqref{thm: min come from line two}, there exists $v\in P_l(\Lambda)$ such that $\KP_K(\Lambda)e$ is isomorphic to $\KP_K(\Lambda)p_v$ as a left $\KP_K(\Lambda)$-module.  Lemma~\ref{lem: line gives K} says that  $p_v\KP_K(\Lambda)p_v$ is isomorphic to $Kp_v$ as a ring; since $Kp_v$ and $K$ are isomorphic,  so are $p_v\KP_K(\Lambda)p_v$ and  $K$.

For an idempotent $f$ in a ring $R$, the map $\phi: fRf\to \End_R(Rf)$, where $\phi(frf)(tf)=tfrf$, is a ring isomorphism. Thus 
\[
e\KP_K(\Lambda)e\cong \End_{\KP_K(\Lambda)}(\KP_K(\Lambda)e)\cong \End_{\KP_K(\Lambda)}(\KP_K(\Lambda)p_v)\cong p_v\KP_K(\Lambda)p_v\cong K.\]
 The converse is obvious.
 \end{proof}

To state our main theorem we need a few definitions.  Suppose $V$ is a cancellative abelian monoid.  An element $u\in V$ is a \emph{unit} if there exists $v\in V$ such that $u+v=0$. An element $s\in V$  is an \emph{atom} (or irreducible) if $s=u+v$ implies $u$ or $v$ is a unit; $V$ is \emph{atomic} if every element of the monoid is a sum of atoms, and $V$ satisfies the \emph{refinement property} if  $a+b=c+d$  in $V$ implies there exist $z_1,z_2,z_3,z_4\in V$ such that $a=z_1+z_2$, $b=z_3+z_4$, $c=z_1+z_3$ and $d=z_2+z_4$.

 For a ring $R$, denote by  $V(R)$  the monoid of finitely generated projective modules over $R$.   The direct sum of nonzero projective modules is nonzero, so the only unit in $V(R)$ is $0$. Thus  $s\in V(R)$ is an atom if $s=u+v$ implies that either $u$ or $v$ is zero.  
Let $e$ and $f$ be idempotents in $R$. Since $R\cong Re\oplus R(1-e)$, $Re$ is a projective module.  We say $e$ and $f$ are \emph{equivalent} if $Re\cong Rf$, and we denote the equivalence class of $e$ in $R$ by $[e]$. If there exist $a,b\in R$ such that $e=ab$ and $f=ba$ then the map $Re\to Rf: re\mapsto rea$ is an isomorphism of $R$-modules with inverse $rf\mapsto rfb$, and hence $e$ and $f$ are equivalent.  We identify the set of equivalence classes of idempotents in $R$ with a submonoid of $V(R)$ via $[e]\mapsto Re$.

Let $R$ be a ring with local units.  There is an order on idempotents in $R$ where $e\leq f$ if and only if $ef=e$.   We claim that  $e\leq f$ if and only if $Re\subset Rf$. If $e\leq f$ then $aef=ae$ for all $a\in R$, so $Re\subset Rf$.  Conversely, if $Re\subset Rf$ then $e\in Rf$ since $R$ has local units. So $e=af$ for some $a\in R$ and hence $ef=aff=af=e$. Thus  $e\leq f$, completing the claim.  

A ring $R$ is \emph{von Neumann regular} if for every $a\in R$ there exists a $b\in R$ such that $aba=a$.  We need the following lemma to prove Theorem~\ref{thm: KP semisimple}. 

\begin{lem}\label{lem: R A give one inf} If $V(\KP_K(\Lambda))$ is an atomic monoid with the refinement property then for every $y\in \Lambda^\infty$ there exists $m\in \N^k$ such that $y(m)\Lambda^\infty=\{\sigma^m(y)\}$.\end{lem}

\begin{proof} By way of contradiction, suppose  there exists $y\in \Lambda^\infty$ such that $|y(m)\Lambda^\infty|>1$ for all $m\in \N^k$. Let $v=y(0)$.  Using Remark~\ref{many remarks}\eqref{many remarks one},  there exists a sequence $\{m_i\}\subset\N^k$ such that $m_{i+1}\geq m_i$ and $|y(m_i)\Lambda^{m_{i+1}-m_i}|>1.$  For $i\in\N\setminus\{0\}$ set
\[e_i:=s_{y(0,m_i)}s_{y(0,m_i)^*}.\]
We claim $\{p_v\KP_K(\Lambda)e_i\}$ is a strictly decreasing sequence of left ideals in $p_v \KP_K(\Lambda)p_v$.

Let $i\in\N\setminus\{0\}$. To see $p_v\KP_K(\Lambda)e_i$ contains $p_v\KP_K(\Lambda)e_{i+1}$ note that
\begin{align*}
e_{i+1}e_i
&=s_{y(0,m_{i+1})}s_{y(0,m_{i+1})^*}s_{y(0,m_i)}s_{y(0,m_i)^*}\\
&=s_{y(0,m_{i+1})}s_{y(m_i, m_{i+1})^*}s_{y(0,m_{i})^*}s_{y(0,m_i)}s_{y(0,m_i)^*}\\
&=s_{y(0,m_{i+1})}s_{y(m_i, m_{i+1})^*}s_{y(0,m_i)^*}\quad\text{(using (KP3))}\\
&=s_{y(0,m_{i+1})}s_{y(0,m_{i+1})^*}=e_{i+1} \quad\text{(using (KP2))}.
\end{align*}
This shows that $e_i\geq e_{i+1}$ and that $p_v\KP_K(\Lambda)e_i\supset p_v\KP_K(\Lambda)e_{i+1}$.  To see the  containment is strict we will show that  $e_{i}\notin p_v\KP_K(\Lambda)e_{i+1}$. Since $|y(m_i)\Lambda^{m_{i+1}-m_i}|>1,$ we can pick $\mu_i\in y(m_i)\Lambda^{m_{i+1}-m_i}\setminus\{y(m_i, m_{i+1})\}$.  Consider $\nu_i=y(0,m_i)\mu_i$.  Then, using (KP3), 
\[e_is_{\nu_i}=s_{y(0,m_i)}s_{y(0,m_i)^*}s_{\nu_i}=s_{y(0,m_i)}s_{\mu_i}=s_{\nu_i}\neq 0.\]
But for any $a\in\KP_K(\Lambda)$
\[p_vae_{i+1} s_{\nu_i}=p_vas_{y(0,m_{i+1})}s_{y(m_i, m_{i+1})^*}s_{y(0,m_{i})^*}s_{y(0,m_i)}s_{\mu_i}=0\]
by (KP3) because $\mu_i$ has degree $m_{i+1}-m_i$.  Thus $p_vae_{i+1}\neq e_i$ for all $a\in \KP_K(\Lambda)$ and hence $e_i\notin p_v\KP_K(\Lambda)e_{i+1}$.  Therefore
$p_v\KP_K(\Lambda)e_i\supsetneq p_v\KP_K(\Lambda)e_{i+1}$ and $e_i\gneq e_{i+1}$.  
Now  $e_{i+1}$ and $e_{i}-e_{i+1}$ are nonzero orthogonal idempotents and, denoting an orthogonal sum with $\oplus$, we have   $e_i=e_{i+1}\oplus(e_{i}-e_{i+1}).$  

Since $V(\KP_K(\Lambda))$ is atomic by assumption, there exists $N$ and atoms $q_1,\ldots, q_N\in V(\KP_K(\Lambda))$ so that $[e_1]=q_1+\cdots+q_N$. For this $N$, we have 
\[e_1=e_{N+1}\oplus (e_{N}-e_{N+1})\oplus\cdots\oplus (e_1-e_2).\]
So in $V(\KP_K(\Lambda))$ we have \[[e_{N+1}]+ [e_{N}-e_{N+1}]+\cdots+ [e_1-e_2]=q_1+\cdots+q_N.\]
We note here for future use that $[e_{N+1}]\neq 0$ because
$e_{N+1}=s_{y(0,m_{N+1})}s_{y(0, m_{N+1})^*}$  is equivalent to $s_{y(0, m_{N+1})^*}s_{y(0,m_{N+1})}=p_{y(m_{N+1})}\neq 0$.

By assumption $V(\KP_K(\Lambda))$ has the refinement property, so there exist $z_{j,l}$, $1\leq j\leq N, 1\leq l\leq N+1$ in $V(\KP_K(\Lambda))$ such that 
\[q_j=\sum_l z_{j,l},\quad [e_l-e_{l+1}]=\sum_j z_{j,l}\ \text{for $1\leq l\leq N$}\quad \text{and}\quad [e_{N+1}]=\sum_j z_{j,N+1}.
\]
Since $q_j$ is  an atom for each $j$,  there exists a unique $l_j\in \{1,\ldots N+1\}$ such that $z_{j, l_j}\neq 0$. In other words, there are precisely $N$ nonzero $z_{j,l}$. So there exists an $l_0\in \{1,\ldots, N+1\}$ such that $z_{j,l_0}=0$ for all $j$. But $[e_{N+1}]\neq 0$, so  $l_0\neq N+1$ and $[e_{l_0}-e_{l_0+1}]=0$.  Now 
$\{0\}=\KP_K(\Lambda)=\KP_K(\Lambda)(e_{l_0}-e_{l_0+1})$. By nondegeneracy of $\KP_K(\Lambda)$  we have $e_{l_0}=e_{l_0+1}$,  a contradiction.  Hence there exists $m$ such that $|y(m)\Lambda^\infty|=1$.
\end{proof}

In \cite{APPM10} Abrams, Aranda Pino, Perera and Siles Molina ``characterize the semisimple Leavitt path algebras by describing them in categorical, ring-theoretic, graph-theoretic, and explicit terms.''  They start by proving in \cite[Theorem~2.3]{APPM10} that a semiprime ring $R$ with local units  is semisimple if and only if eight equivalent conditions hold.  Among these conditions are  that $R$ is locally left (respectively, right) artinian, that every corner of $R$ is left (respectively, right) artinian, and that $R$ is isomorphic to a certain direct sum of matrix algebras over corners of $R$. In \cite[Theorem~2.4]{APPM10} they restrict their attention to the Leavitt path algebra $L(E)$ of a directed graph $E$  and link the conditions of \cite[Theorem~2.3]{APPM10} to $L(E)$ being  categorically left (respectively, right) artinian, $L(E)$ being von Neumann regular and properties of the monoid $V(L(E))$,  and purely graph-theoretic conditions. Our  Theorem~\ref{thm: KP semisimple} is the higher-rank analogue of \cite[Theorem~2.4]{APPM10}, but we have omitted some of the many equivalent  conditions of \cite[Theorem~2.4]{APPM10}  to concentrate on those pertinent to $\KP_K(\Lambda)$ and $\Lambda$ in particular. 

\begin{thm}\label{thm: KP semisimple}Let $\Lambda$ be a row-finite $k$-graph with no sources and $K$ be a field.  The following conditions are equivalent.
\begin{enumerate}
\item \label{part ss}$\KP_K(\Lambda)$ is semisimple.
\item \label{part matrix dec} There exists a countable set $\Upsilon$  and for each $i\in\Upsilon$ there exists $n_i\in \N\cup \{\infty\}$ such that $\KP_K(\Lambda)\cong \bigoplus_{i\in \Upsilon} M_{n_i}(K)$.
\item\label{part vN R and N} There exists a countable set $\Upsilon$  and for each $i\in\Upsilon$ there exists $n_i\in \N\cup \{\infty\}$ such that $V(\KP_K(\Lambda))\cong \oplus_{i\in \Upsilon}\N$, and $\KP_K(\Lambda)$ is von Neumann regular.
\item\label{part vN R and atomic} $V(\KP_K(\Lambda))$ is cancellative and atomic, and $\KP_K(\Lambda)$ is von Neumann regular.
\item \label{part graph cond} For every $y\in \Lambda^\infty$, there exists $m\in \N^k$ such that $y(m)\in P_l(\Lambda)$.
\end{enumerate}
\end{thm}

\begin{rmk}\label{rmk inf imply aperiodic}Condition~\eqref{part graph cond} of Theorem~\ref{thm: KP semisimple} implies that every infinite path is aperiodic.  To see this, suppose $y\in \Lambda^\infty$ and pick $m$ such that $y(m)\in P_l(\Lambda)$.  Then $y(m)\Lambda^\infty=\{\sigma^m(y)\}$ with $\sigma^m(y)$ aperiodic.  Since $\sigma^m(y)$ is aperiodic, so is $y$. \end{rmk}

\begin{proof}[Proof of Theorem~\ref{thm: KP semisimple}] We will show $\eqref{part ss} \Rightarrow \eqref{part matrix dec}\Rightarrow \eqref{part vN R and N}\Rightarrow\eqref{part vN R and atomic}\Rightarrow\eqref{part graph cond}\Rightarrow\eqref{part ss}$.

\smallskip\eqref{part ss} $\Rightarrow$ \eqref{part matrix dec}.  Recall that an idempotent $e$ in a semiprime ring $R$ is minimal if $Re$ is a minimal left-ideal \cite[p.~143]{BMM96}.   Theorem~2.3  of  \cite{APPM10}  says that  a semiprime ring $R$ is  semisimple if and only if there exists an index set $\Upsilon$ and for each $i\in \Upsilon$ there exist minimal idempotents $e_i\in R$ and $n_i\in\N\cup\{\infty\}$ such that $R\cong\sum_{i\in\Upsilon}M_{n_i}(e_iRe_i)$.   Since $Re_i$ is a minimal left-ideal in $R$,  \cite[Proposition~4.3.3]{BMM96} shows that $e_iRe_i$ is a division ring.  Assume $\KP_K(\Lambda)$ is semisimple.  By Lemma~\ref{lem: KP nondegenerate}, $\KP_K(R)$ is semiprime.  Applying \cite[Theorem~2.3]{APPM10} to $R=\KP_K(\Lambda)$ we obtain  $\Upsilon$, $e_i$ and $n_i$ as above. By Lemma~\ref{lem local units}, $\KP_K(\Lambda)$ has a countable set of local units, and thus $\Upsilon$ is countable by \cite[Theorem~2.3]{APPM10}.  By Lemma~\ref{lem: div ring}, each  $e_i\KP_K(\Lambda)e_i$ is isomorphic  to $K$. This gives \eqref{part matrix dec}.

\smallskip
\eqref{part matrix dec} $\Rightarrow$ \eqref{part vN R and N}.  Assume \eqref{part matrix dec}. Each $M_{n_i}(K)$ is von Neumann regular, and hence so is the direct sum. Since $V(M_{n_i}(K))=\N$ and $V(R_1\oplus R_2)=V(R_1)\oplus V(R_2)$, \eqref{part vN R and N} holds.

\smallskip  \eqref{part vN R and N} $\Rightarrow$ \eqref{part vN R and atomic}. This follows because $\oplus_{i\in \Upsilon}\N$  cancellative and atomic.

\smallskip
\eqref{part vN R and atomic} $\Rightarrow$ \eqref{part graph cond}.
By \cite[3rd example on p.~412]{Ara97}, ``$\pi$-regular rings'' are exchange rings.  Since $\pi$-regular rings are a generalization of von Neumann regular rings it follows that von Neumann regular rings are exchange rings.  By \cite[Proposition~1.2]{AGOP98}, if $R$ is an exchange ring then $V(R)$ has the refinement property.  

Assume \eqref{part vN R and atomic}. Then $V(\KP_K(\Lambda))$ has the refinement property. Let $y\in \Lambda^\infty$.  By Lemma~\ref{lem: R A give one inf},  there exists $t\in \N^k$ such that $y(t)\Lambda^\infty=\{\sigma^t(y)\}$.
We claim that  $y(t)$ is a line point.

By way of contradiction,  suppose that $y(t)$ is not a line point. Since $y(t)\Lambda^\infty=\{\sigma^t(y)\}$, $\sigma^t(y)$  must be periodic. For convenience set $z:=y(t,\infty)$; then $z(0)\Lambda^\infty=\{z\}$ and $z$ is periodic. So there exists $m\neq n\in \N^k$ such that $z(m)=z(n)$.

 Let  $\pi_h:\N^k\to \N$  be the projection onto the $h$-th factor. By switching $m$ and $n$ if necessary, we may  assume that  there exists an $l\in \{1,\ldots, k\}$ such that $\pi_l(m-n)>0$.  Let $v=z(0)$. Consider the element $p_v+s_{z(0,m)}s_{z(0,n)^*}$.  Since we assumed that $\KP_K(\Lambda)$ is von Neumann regular, there exists $b\in \KP_K(\Lambda)$ such that 
\begin{equation}\label{eq vN reg compare}
(p_v+s_{z(0,m)}s_{z(0,n)^*})b(p_v+s_{z(0,m)}s_{z(0,n)^*})=p_v+s_{z(0,m)}s_{z(0,n)^*}.
\end{equation}
Since both sides of \eqref{eq vN reg compare} are nonzero we can assume $b\in p_v \KP_K(\Lambda)p_v$.  Write $b$ in normal form as $b=\sum_{(\alpha,\beta)\in F}r_{\alpha,\beta}s_{\alpha}s_{\beta^*}$. Then for all $(\alpha,\beta),(\mu,\nu)\in F$  we have $d(\beta)=d(\nu)$, and, since $|v\Lambda^\infty|=1$, we have $\beta=\nu$.  So we can assume that 
$$b=\sum_{i=1}^N r_i s_{\alpha_i}s_{\beta^*}\quad r_i\in K\setminus \{0\}$$
for some fixed $\beta$.  Furthermore, we can  assume that the $\alpha_i$ are ordered so that $\pi_l(d(\alpha_i))\leq \pi_l(d(\alpha_{i'}))$ for all $1\leq i\leq i'\leq N$.

We will obtain a contradiction by comparing the degrees on the left and right  sides of \eqref{eq vN reg compare}. Define the $\pi_l$-degree of a homogeneous element to be the $l$-th component of its degree.  The minimal $\pi_l$-degree associated to the left side of \eqref{eq vN reg compare} is obtained by the monomial $p_vs_{\alpha_1}s_{\beta^*}p_v=s_{\alpha_1}s_{\beta^*}$.  The minimal $\pi_l$-degree appearing in the right side of \eqref{eq vN reg compare} is zero.  So $\pi_l(d(\alpha_1)-d(\beta))=0$.   The maximal $\pi_l$-degree on the left hand side of \eqref{eq vN reg compare} is obtained by the monomial $s_{z(0,m)}s_{z(0,n)^*}s_{\alpha_N}s_{\beta^*}s_{z(0,m)}s_{z(0,n)^*}$ whose $\pi_l$-degree is $2\pi_l(m-n)+\pi_l(d(\alpha_N)-d(\beta))$.  By Lemma~\ref{lem: nonzero}, $s_{z(0,m)}s_{z(0,n)^*}s_{\alpha_N}s_{\beta^*}s_{z(0,m)}s_{z(0,n)^*}\neq 0$. Since the maximal $\pi_l$-degree on the right hand side of \eqref{eq vN reg compare} is $\pi_l(m-n)>0$, we have $2\pi_l(m-n)+\pi_l(d(\alpha_N)-d(\beta))=\pi_l(m-n)$.  Thus $\pi_l(d(\alpha_N)-d(\beta))=-\pi_l(m-n)<0$.  Now $0=\pi_l(d(\alpha_1)-d(\beta))\leq\pi_l(d(\alpha_N)-d(\beta))<0$, a contradiction.  Thus $z=\sigma^t(y)$ is aperiodic, and $y(t)$ is a line point as claimed.  This gives \eqref{part graph cond}.

\smallskip
\eqref{part graph cond} $\Rightarrow$ \eqref{part ss}.   Assume \eqref{part graph cond}.   By Corollary~\ref{cor: 0 soc plus}, we need  to show that $\overline{P_l(\Lambda)}=\Lambda^0$.  
By way of contradiction suppose there exists $v\in \Lambda^0\setminus \overline{P_l(\Lambda)}$.  We define a path $y\in v\Lambda^\infty$ by recursion.  Since $v$ is in the complement of the saturated set $\overline{P_l(\Lambda)}$, there exists a path $\mu_1\in v\Lambda^{(1,\ldots ,1)}$ such that $s(\mu_1)\notin  \overline{P_l(\Lambda)}$.  Suppose that $\mu_i$ has been defined is such a way that $r(\mu_i)=s(\mu_{i-1})$, $d(\mu_i)=(1,\ldots, 1)$ and $s(\mu_i)\notin \overline{P_l(\Lambda)}$.  Because $\overline{P_l(\Lambda)}$ is saturated, there exists  $\mu_{i+1}\in s(\mu_i)\Lambda^{(1,\ldots,1)}$ such that $s(\mu_{i+1})\notin\overline{P_l(\Lambda)}$.  Set $y=\mu_1\mu_2\cdots\in v\Lambda^\infty$.  By \eqref{part graph cond}, there exists $m\in \N^k$ such that $y(m)\in P_l(\Lambda)$.  But for $i$ sufficiently large, $i\cdot (1,\ldots, 1)>m$ and since $P_l(\Lambda)$ is hereditary,  $s(\mu_{i})\in P_l(\Lambda)$,  a contradiction.  Thus $\Lambda^0=\overline{P_l(\Lambda)}$ and hence $\KP_K(\Lambda)$ is semisimple.\end{proof}

\begin{rmk}
A simple Kumjian-Pask algebra need not be semisimple.  To see this,  pick any row-finite, aperiodic and cofinal $k$-graph $\Lambda$ which contains a periodic path. Then $\KP_K(\Lambda)$ is simple by \cite[Theorem~6.1]{A-PCaHR11}, but is not semisimple by Theorem~\ref{thm: KP semisimple}.
\end{rmk}

Our next goal is to find a more precise description of the matrix decomposition of  semisimple Kumjian-Pask algebras given by item~\eqref{part matrix dec} of Theorem~\ref{thm: KP semisimple}. We will show that $\KP_K(\Lambda)$ is isomorphic to $\bigoplus_{[v]\in P_l(\Lambda)/\sim} I_{\overline{\{v\}}}$ where $\sim$ is an equivalence relation on the set of line points and $I_{\overline{\{v\}}}$ is the graded ideal associated to the saturated hereditary  subset $\overline{\{v\}}$ of $\Lambda^0$.

Let $v,w\in P_l(\Lambda)$ so that $v\Lambda^\infty=\{x\}$ and  $w\Lambda^\infty=\{y\}$. Define a relation on $P_l(\Lambda)$ by 
$$v\sim w \Leftrightarrow \exists ~ m,n\in\N^k\text{ such that } x(m)=y(n).$$  It follows from item~\eqref{prop: line equiv disjoint} of Lemma~\ref{lem: lin pt ideal} that $\sim$ defines an equivalence relation on $P_l(\Lambda)$.

Let $H$ be a saturated hereditary  subset of $\Lambda^0$ and recall that the \emph{quotient graph} $\Lambda\setminus H$ is
\begin{equation}\Lambda\setminus H:=(\Lambda^0\setminus H, s\inv(\Lambda^0\setminus H), r|,s|);\label{eq: quot graph}
\end{equation}
$\Lambda\setminus H$ is a row-finite $k$-graph with no sources by \cite[Lemma~5.3]{A-PCaHR11}. 

\begin{lem}\label{lem: lin pt ideal} Let  $v, w\in P_l(\Lambda)$, say $v\Lambda^\infty=\{y\}$ and $w\Lambda^\infty=\{z\}$.  
\begin{enumerate}
\item \label{lin ideal cont}If $u\in \overline{\{v\}}$, then $\overline{\{u\}}=\overline{\{v\}}.$  
\item \label{prop: line equiv disjoint}
$\overline{\{w\}}\cap \overline{\{v\}}\neq \emptyset\quad\Leftrightarrow\quad \overline{\{w\}}= \overline{\{v\}}\quad \Leftrightarrow \quad w\sim v.$
\item\label{lem: lin ideal min}If $\Lambda$ contains no periodic paths then the ideal $I_{\overline{\{v\}}}$ is minimal.
\end{enumerate}
\end{lem}

\begin{proof} \eqref{lin ideal cont} Suppose $u\in \overline{\{v\}}$. Then $\overline{\{u\}}\subset \overline{\{v\}}$.  For the reverse inclusion it suffices to show that $v\in\overline{\{u\}}$. Recall that $\Sigma^0(\{v\})$ is the hereditary closure of $\{v\}$. Since $u\in \overline{\{v\}}$, by Lemma~\ref{oldRemark}, there exists $\mu\in u\Lambda\Sigma^0(\{v\})$. But $\overline{\{u\}}$ is hereditary, so $s(\mu)\in\overline{\{u\}}\cap \Sigma^0(\{v\})$.
Because $\Sigma^0(\{v\})$ is the hereditary closure of $\{v\}$  there exists a path $\lambda\in v\Lambda s(\mu)\subset v\Lambda\overline{\{u\}}$.  Since $v\in P_l(\Lambda)$, $|v\Lambda^{d(\lambda)}|=1$  by Remark~\ref{many remarks}\eqref{many remarks one}. That is, $v\Lambda^{d(\lambda)}=\{\lambda\}$.  Hence $s(v\Lambda^{d(\lambda)})\subset \overline{\{u\}}$ which implies $v\in\overline{\{u\}}$ because $\overline{\{u\}}$ is saturated.

\smallskip

\eqref{prop: line equiv disjoint}
First, assume $\overline{\{w\}}\cap \overline{\{v\}}\neq \emptyset$. Let $u\in \overline{\{w\}}\cap \overline{\{v\}}$. By \eqref{lin ideal cont}, we have $\overline{\{w\}}=\overline{\{u\}}=\overline{\{v\}}$.

Second, assume $\overline{\{w\}}=\overline{\{v\}}$. Then $v\in \overline{\{w\}}$, so there exists $\mu\in v\Lambda\Sigma^{0}(\{w\})$ by Lemma~\ref{oldRemark}. Since $s(\mu)\in \Sigma^{0}(\{w\})$, there exists $\nu\in w\Lambda s(\mu)$.  Now $v\Lambda^\infty=\{y\}$ and $w\Lambda^\infty=\{z\}$ so $\mu=y(0,d(\mu))$ and $\nu=z(0,d(\nu))$.   But $y(d(\mu))=s(\mu)=s(\nu)=z(d(\nu))$, and hence $v\sim w$.

Third, assume $v\sim w$. There exist $n,m\in\N^k$ such that $y(m)=z(n)$. But $\overline{\{v\}}$ and $\overline{\{w\}}$ are hereditary, so $y(m)\in \overline{\{v\}}\cap \overline{\{w\}}\neq \emptyset$.

\smallskip

\eqref{lem: lin ideal min} Now suppose that $\Lambda$ contains no periodic paths. Then $\Lambda\setminus H$  is aperiodic for all saturated hereditary  subsets $H$ of $\Lambda^0$, and \cite[Theorem~5.6]{A-PCaHR11} implies that every ideal of $\KP_K(\Lambda)$ is graded. Thus $H\mapsto I_H$ is an isomorphism of the lattice of saturated hereditary  subsets of $\Lambda$ onto the lattice of ideals of $\KP_K(\Lambda)$. By \eqref{lin ideal cont}, $\overline{\{v\}}$ is a minimal saturated hereditary  subset, and hence the ideal $I_{\overline{\{v\}}}$ is a minimal ideal.
\end{proof}

Next we want to show that the ideals $\{I_{\overline{\{v\}}}\}_{[v]\in P_l(\Lambda)/\sim}$ are mutually orthogonal.   
For ideals $I,J$ in a ring $R$ we let $IJ:=\{ab:a\in I, b\in J\}$ and $I\cdot J:=\spn\{ab:a\in I, b\in J\}.$ Then $I\cdot J$ is a two-sided ideal in $R$,  and $IJ\subset I\cdot J \subset I\cap J$.

\begin{lem}\label{lem: ideal intersections} Let $H$ and $L$ be saturated hereditary subsets of $\Lambda^0$.  Then $I_H\cap I_L=I_{H\cap L}=I_H\cdot I_L$.  In particular, if $H\cap L=\emptyset$ then $I_HI_L=\{0\}$. 
\end{lem}

\begin{proof} The second statement follows immediately from the first because $I_HI_L\subset I_H\cdot I_L=I_{H\cap L}$.  Recall that $H\mapsto I_H$ is an isomorphism of the lattice of saturated hereditary  subsets of $\Lambda$ onto the lattice of graded ideals of $\KP_K(\Lambda)$. Since the intersection of graded ideals is graded, $I_H\cap I_L=I_W$ for some saturated hereditary  subset $W$ of $\Lambda^0$.  Using that $I_W$ is generated by $p_v$ with $v\in W$, it is straightforward to show that $W=H\cap L$. Thus $I_H\cap I_L=I_{H\cap L}$

By definition of $I_H\cdot I_L$ we have $I_H\cdot I_L\subset I_H\cap I_L=I_{H\cap L}$.  For the reverse inclusion, let $p_v\in I_{H\cap L}$. Then $v\in H\cap L$, and hence $p_v\in I_H$ and $p_v\in I_L$. Now $p_v=p_v^2\in I_H\cdot I_L$.  Since $I_{H\cap L}$ is generated by $p_v$ with $v\in H\cap L$ and $I_H\cdot I_L$ is an ideal we get $I_{H\cap L}\subset I_H\cdot I_L$.
 \end{proof} 
 
 We now have another corollary of Theorem~\ref{thm: soc cor line}. 

\begin{cor}\label{cor: KP dir sum ideal} Let $\Lambda$ be a row-finite $k$-graph with no sources and let $K$ be a field.  If $\KP_K(\Lambda)$ semisimple then $\KP_K(\Lambda)=\bigoplus_{[v]\in P_l(\Lambda)/\sim} I_{\overline{\{v\}}}$.\end{cor}

\begin{proof} Let $v,w \in P_l(\Lambda)$ such that $v\not\sim w$. By Lemma~\ref{lem: lin pt ideal}\eqref{prop: line equiv disjoint}, ${\overline{\{v\}}}\cap {\overline{\{w\}}}=\emptyset$, so by Lemma~\ref{lem: ideal intersections} we have  $I_{\overline{\{v\}}}I_{\overline{\{w\}}}=\{0\}$.   Thus the internal direct sum $\bigoplus_{[v]\in P_l(\Lambda)/\sim} I_{\overline{\{v\}}}$ is an ideal in $\KP_K(\Lambda) $. 

Since $\KP_K(\Lambda)$ is semisimple, $\Soc(\KP_K(\Lambda))=\KP_K(\Lambda)$. But $\Soc(\KP_K(\Lambda))$ is the ideal generated by $\{p_v:v\in P_l(\Lambda)\}$ by Theorem~\ref{thm: soc cor line}. Thus $\KP_K(\Lambda)\subset \bigoplus_{[v]\in P_l(\Lambda)/\sim} I_{\overline{\{v\}}}$, that is, $\KP_K(\Lambda)=\bigoplus_{[v]\in P_l(\Lambda)/\sim} I_{\overline{\{v\}}}$.
\end{proof}

Next we analyze the ideals $I_{\overline{\{v\}}}$ for $v\in P_l(\Lambda)$.

\begin{lem}\label{lem minfty} 
\begin{enumerate}
\item\label{lem: contain inf mat}
If  $v\in P_l(\Lambda)$ then $I_{\overline{\{v\}}}$ contains a subalgebra isomorphic to $M_\infty(K)$.
\item\label{cor: simple and ss} 
If $\KP_K(\Lambda)$ is both simple and semisimple then $\KP_K(\Lambda)\cong M_\infty(K)$.
\item\label{prop line Minfty} 
If $\KP_K(\Lambda)$ is semisimple then  $I_{\overline{\{v\}}}\cong M_\infty(K)$ for all $v\in P_l(\Lambda)$. 
\end{enumerate}
\end{lem}

\begin{proof} \eqref{lem: contain inf mat} Let $v\in P_l(\Lambda)$.  Then $v\Lambda^\infty=\{y\}$. It suffices to show $I_{\overline{\{v\}}}$ contains a set of matrix units indexed by $\N\times \N$.     For $i,j\in \N$ define
$$e_{i,j}=\begin{cases} s_{y((i,0,\ldots,0), (j,0,\ldots,0))} & \text{if~} i<j\\
s_{y((j,0,\ldots,0), (i,0,\ldots,0))^*} & \text{if~} j<i\\
p_{y(j,0,\ldots,0)} & \text{if~} i=j.\end{cases}$$

We claim $\{e_{i,j}\}$ forms a set of matrix units in $I_{\overline{\{v\}}}$.  First note that $e_{i,j}e_{h,\ell}=0$ unless $j=h$.    A case by case analysis shows $e_{i,j}e_{j,\ell}=e_{i,\ell}.$  For example: 
suppose $i<\ell<j$. Then $|v\Lambda^{j-\ell}|=1$ since  $v\in P_l(\Lambda)$.  Thus
\begin{align*}e_{i,j}e_{j,\ell}&= s_{y((i,0,\ldots,0), (j,0,\ldots,0))} s_{y((\ell,0,\ldots,0), (j,0,\ldots,0))^*}\\
&{=}s_{y((i,0,\ldots,0), (\ell,0,\ldots,0))}s_{y((\ell,0,\ldots,0), (j,0,\ldots,0))} s_{y((\ell,0,\ldots,0), (j,0,\ldots,0))^*}\\
\intertext{which by Remark~\ref{many remarks}\eqref{many remarks four} is}
&{=}s_{y((i,0,\ldots,0), (\ell,0,\ldots,0))}p_{y(\ell,0,\ldots,0)}=e_{i,\ell}.\end{align*}
The other cases follow from similar arguments.  Thus $\{e_{i,j}\}_{i,j\in \N}$ forms a set of matrix units in $I_{\overline{\{v\}}}$, and hence $M_\infty(K)$ is isomorphic to a subalgebra of $I_{\overline{\{v\}}}$. 

\eqref{cor: simple and ss} Suppose that $\KP_K(\Lambda)$ is semisimple. Then $\KP_K(\Lambda)\cong \bigoplus_{i\in \Upsilon} M_{n_i}(K)$ by Theorem~\ref{thm: KP semisimple}($\eqref{part ss}\Rightarrow \eqref{part matrix dec}$).  Since $\KP_K(\Lambda)$ is simple this direct sum has only one summand $M_m(K)$ for some $m\in \N\cup\{\infty\}$.    By \eqref{lem: contain inf mat},   $M_\infty(K)$ is a subalgebra of $\KP_K(\Lambda)$. Since $M_\infty(K)$ cannot be isomorphic to a subalgebra of $M_n(K)$ for $n\in \N$, $m=\infty$ as desired.

\eqref{prop line Minfty} Suppose that $\KP_K(\Lambda)$ is semisimple. We will show that $I_{\overline{\{v\}}}$ is isomorphic to a simple Kumjian-Pask algebra and then invoke \eqref{cor: simple and ss} to get the result. Since $\KP_K(\Lambda)$ is semisimple it follows from Corollary~\ref{cor: KP dir sum ideal} that $I_{\overline{\{v\}}}\cong \KP_K(\Lambda)/J$ where 
\[J=\bigoplus_{[w]\in (P_l(\Lambda)/\sim)\setminus\{[v]\}}I_{\overline{\{w\}}}.
\] It follows from Theorem~\ref{thm: KP semisimple}\eqref{part graph cond}   that $\Lambda$ has no periodic paths (see Remark~\ref{rmk inf imply aperiodic}).  So for any saturated hereditary set $H$, $\Lambda\setminus H$ is aperiodic, and therefore every ideal of $\KP_K(\Lambda)$ is graded by \cite[Theorem~5.6]{A-PCaHR11}.   In particular   $J$ is graded, and hence $J=I_L$ for some saturated hereditary subset $L$ of $\Lambda^0$ by \cite[Theorem~5.1]{A-PCaHR11}. Now by \cite[Proposition~5.5]{A-PCaHR11}, $I_{\overline{\{v\}}}\cong \KP_K(\Lambda\setminus L)$.

Next we will show that $\KP_K(\Lambda\setminus L)$ is simple and semisimple.  
Since $\KP_K(\Lambda)$ is semisimple 
every infinite path $y\in\Lambda^\infty$ contains a line point by Theorem~\ref{thm: KP semisimple}. But now every infinite path $y\in(\Lambda\setminus L)^\infty$ contains a line point as well, so $\KP_K(\Lambda\setminus L)$ is semisimple by Theorem~\ref{thm: KP semisimple}.

As observed above,  $\Lambda\setminus L$ is aperiodic. Theorem~6.1 of \cite{A-PCaHR11} says that $\KP_K(\Lambda\setminus L)$ is simple if and only if $\Lambda\setminus L$ is cofinal.
  Pick $x\in (\Lambda\setminus L)^\infty$, $w\in (\Lambda\setminus L)^0$ and $y\in w(\Lambda\setminus L)^\infty$.  We need to show  that there exists a $t$ such that $y(0,t)$ connects $w$ to $x$. 

Since we know that $x,y$ both contain line points, there exists $m,n\in\N^k$ such that $x(m), y(n)\in P_l(\Lambda)$.  Since $x(m), y(n)\notin L$ we have $p_{x(m)}, p_{y(n)}\notin J$, and hence $x(m), y(n)\in \overline{\{v\}}$.  Thus $x(m)\sim y(n)$ by Lemma~\ref{lem: lin pt ideal}\eqref{prop: line equiv disjoint}. By definition of $\sim$ there exist $m',n'$ such that  $x(m+m')=\sigma^m(x)(m')=\sigma^n(y)(n')=y(n+n')$. Now $y(0, n+n')$ connects $w$ to $x$ and hence $\Lambda\setminus L$ is cofinal.  Thus $\KP_K(\Lambda\setminus L)$ is simple. Now we apply \eqref{cor: simple and ss} to see $I_{\overline{\{v\}}}\cong \KP_K(\Lambda\setminus L)\cong M_\infty (K)$.\end{proof}

The next result sharpens  condition~\eqref{part matrix dec} of Theorem~\ref{thm: KP semisimple}.

\begin{thm}\label{thm: KP ss matrix}Let $\Lambda$ be a row-finite $k$-graph with no sources and $K$ be a field.  Then $\KP_K(\Lambda)$ is semisimple if and only if $\KP_K(\Lambda)\cong \bigoplus_{[v]\in (P_l(\Lambda)/\sim)} M_\infty(K)$. \end{thm}

\begin{proof}Suppose that $\KP_K(\Lambda)$ is semisimple. Then Corollary~\ref{cor: KP dir sum ideal} gives that  $\KP_K(\Lambda)=\bigoplus_{[v]\in (P_l(\Lambda)/\sim)} I_{\overline{\{v\}}}$ and so   $\KP_K(\Lambda)\cong \bigoplus_{[v]\in (P_l(\Lambda)/\sim)} M_\infty(K)$ by Lemma~\ref{lem minfty}\eqref{prop line Minfty}.  The converse is immediate from Theorem~\ref{thm: KP semisimple} ($\eqref{part matrix dec}\Rightarrow \eqref{part ss}$).\end{proof}

\begin{rmk}\label{rmk-ss} Theorem~\ref{thm: KP ss matrix} shows that every semisimple Kumjian-Pask algebra can be obtained as a Leavitt path algebra, that is,  the added generality of $k$-graphs gives nothing new for semisimple algebras. For $n\in \N\cup \{\infty\}$ it is easy to construct a $1$-graph $\Lambda_n$ whose Kumjian-Pask algebra is isomorphic to $\bigoplus_{i=1}^n M_\infty (K)$. Just take $\Lambda_n$ to be the following graph with $n$ vertical components.  
$$
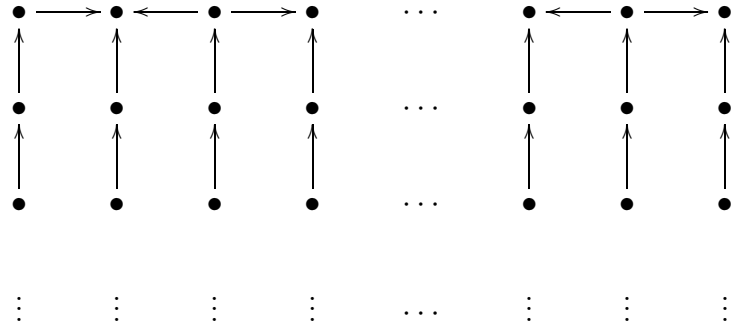
\xymatrix{\bullet \ar[r] &\bullet &\bullet \ar[l]\ar[r]& \bullet &\cdots&\bullet &\bullet \ar[l]\ar[r] &\bullet\\
						\bullet \ar[u] &\bullet \ar[u]  &\bullet \ar[u]&\bullet \ar[u]&\cdots &\bullet \ar[u]&\bullet \ar[u] & \bullet\ar[u] \\
						\bullet \ar[u] &\bullet \ar[u]  &\bullet \ar[u]&\bullet \ar[u]&\cdots &\bullet \ar[u]&\bullet \ar[u] &\bullet\ar[u] \\
						\vdots &\vdots &\vdots & \vdots & \cdots &\vdots &\vdots &\vdots }$$
By inspection $\overline{P_l(\Lambda_n)}=\Lambda^0_n$ and $v\sim w\in P_l(\Lambda)$ if and only if $v$ and $w$ are in the same vertical component.  Thus by Theorem~\ref{thm: KP ss matrix}, $\KP_K(\Lambda_n)\cong\bigoplus_{i=1}^n M_\infty (K)$.
\end{rmk}


\end{document}